\documentclass[12pt]{article}

\title{Torsion contact forms in three dimensions have two or infinitely many Reeb orbits}
\author{Dan Cristofaro-Gardiner\footnote{Partially supported by NSF grant DMS-1402200.}, Michael Hutchings\footnote{Partially supported by NSF grant DMS-1406312 and a Simons Fellowship.}, Daniel Pomerleano\footnote{Partially supported by EPSRC grant EP/L018772/1.}}

\usepackage{amssymb}
\usepackage{latexsym}
\usepackage{amsmath}
\usepackage{amsthm}
\usepackage{amscd}
\usepackage{mathabx}
\usepackage{graphicx}
\usepackage{enumerate}
\usepackage{float}
\usepackage{xcolor}

\numberwithin{equation}{section}

\newtheorem{theorem}{Theorem}[section]
\newtheorem{question}[theorem]{Question}
\newtheorem{proposition}[theorem]{Proposition}
\newtheorem{corollary}[theorem]{Corollary}
\newtheorem{lemma}[theorem]{Lemma}

\newtheorem{lemma-definition}[theorem]{Lemma-Definition}

\theoremstyle{definition}
\newtheorem{definition}[theorem]{Definition}
\newtheorem{remark}[theorem]{Remark}

\newcommand{\eqdef}{\;{:=}\;}

\renewcommand{\frak}{\mathfrak}

\newcommand{\Q}{{\mathbb Q}}
\newcommand{\R}{{\mathbb R}}
\newcommand{\N}{{\mathbb N}}
\newcommand{\Z}{{\mathbb Z}}

\newcommand{\op}{\operatorname}

\newcommand{\M}{\mc{M}}

\newcommand{\Ker}{\op{Ker}}

\newcommand{\tensor}{\otimes}

\renewcommand{\epsilon}{\varepsilon}

\newcommand{\mc}[1]{{\mathcal #1}}
\newcommand{\floor}[1]{\left\lfloor #1 \right\rfloor}
\newcommand{\ceil}[1]{\left\lceil #1 \right\rceil}
\newcommand{\CZ}{\op{CZ}}

\begin{document}

\maketitle

\begin{abstract}
We prove that every nondegenerate contact form on a closed connected three-manifold, such that the associated contact structure has torsion first Chern class, has either two or infinitely many simple Reeb orbits. By previous results it follows that under the above assumptions, 
there are infinitely many simple Reeb orbits if the three-manifold is not the three-sphere or a lens space.  
We also show that for non-torsion contact structures, every nondegenerate contact form has at least four simple Reeb orbits.

\end{abstract}

\section{Introduction} 
 \label{sec:intro}

\subsection{Statement of the main result}

Let $Y$ denote a closed connected three-manifold.
Recall that a {\bf contact form\/} on $Y$ is a $1$-form $\lambda$ on $Y$ such that $\lambda\wedge d\lambda \neq 0$ everywhere. Associated to $\lambda$ is the {\bf Reeb vector field\/} $R$ characterized by $d\lambda(R,\cdot)=0$ and $\lambda(R)=1$.  Also associated to $\lambda$ is the {\bf contact structure\/} $\xi=\Ker(\lambda)$; this is a $2$-plane field on $Y$ which is oriented by $d\lambda$.

A {\bf Reeb orbit\/} is a periodic orbit of $R$, that is a map $\gamma:\R/T\Z\to Y$ for some $T>0$ such that $\gamma'(t)=R(\gamma(t))$ for all $t$. We consider two Reeb orbits to be equivalent if they differ by precomposition with a translation of the domain. A Reeb orbit $\gamma$ is {\bf simple\/} if the map $\gamma$ is an embedding. Every Reeb orbit is an $m$-fold cover of a simple Reeb orbit for some positive integer $m$.

The three-dimensional case of the Weinstein conjecture asserts that every contact form on a closed three-manifold has at least one Reeb orbit.  This was proved by Taubes \cite{taubes-weinstein} in 2006, and various special cases had been proved earlier, see e.g.\ the survey \cite{tw}.

Later, the first two authors established the following refinement of Taubes's result:

\begin{theorem}
\label{thm:two}
\cite{two}
Every contact form on a closed three-manifold has at least two simple Reeb orbits.
\end{theorem}

The lower bound of two is the best possible, because there exist contact forms on $S^3$ with exactly two simple Reeb orbits, see e.g.\ \cite[Ex.\ 1.8]{bn}. One can also take quotients of these examples by cyclic group actions to obtain contact forms on lens spaces with exactly two simple Reeb orbits. However one could try to prove the existence of more simple Reeb orbits under additional assumptions.

The following theorem provides some inspiration.  Recall that if $\gamma$ is a Reeb orbit, the linearized Reeb flow along $\gamma$ defines a symplectic linear map $P_\gamma$, the ``linearized return map'',
from $(\xi_{\gamma(0)},d\lambda)$ to itself. The Reeb orbit $\gamma$ is {\bf nondegenerate\/} if $1$ is not an eigenvalue of $P_\gamma$. In this case, we say that $\gamma$ is {\bf positive hyperbolic\/} if $P_\gamma$ has positive eigenvalues, and {\bf negative hyperbolic\/} if $P_\gamma$ has negative eigenvalues; otherwise $P_\gamma$ has eigenvalues on the unit circle and we say that $\gamma$ is {\bf elliptic\/}. The contact form $\lambda$ is called nondegenerate if every (not necessarily simple) Reeb orbit is nondegenerate.

\begin{theorem}
\label{thm:hwz}
(Hofer-Wyoscki-Zehnder \cite[Cor.\ 1.10]{hwz2})
Let $\lambda$ be a nondegenerate contact form on $S^3$.  Assume that 
\begin{enumerate}[(a)]
\item  $\xi=\Ker(\lambda)$ is the standard\footnote{Here, the ``standard'' contact structure refers to the kernel of the restriction of $\lambda_{std} =  \frac{1}{2} \sum_{i=1}^2 x_i dy_i - y_i dx_i$ to the unit sphere in $\mathbb{C}^2 = \mathbb{R}^4$.} contact structure on $S^3$. 
\item The stable and unstable manifolds of all hyperbolic Reeb orbits of $\lambda$ intersect transversely.
\end{enumerate}
Then $\lambda$ has either two or infinitely many simple Reeb orbits.
\end{theorem}

For more complicated three-manifolds, Colin-Honda \cite{ch} used linearized contact homology to show that for many contact three-manifolds $(Y,\xi)$ supported by an open book decomposition with pseudo-Anosov monodromy, every (possibly degenerate) contact form $\lambda$ with $\Ker(\lambda)=\xi$ has infinitely many simple Reeb orbits.

In fact, no example is currently known of a contact form on a closed connected three-manifold with more than two but only finitely many simple Reeb orbits. Thus it is natural to ask:

\begin{question}
\label{que:2infinity}
Does every contact form on a closed connected three-manifold have either two or infinitely many simple closed orbits?    
\end{question} 

Our main result answers this question in many cases:

\begin{theorem}
\label{thm:main}
Let $Y$ be a closed connected three-manifold and let $\lambda$ be a nondegenerate contact form on $Y$.  Assume that $c_1(\xi)\in H^2(Y;\Z)$ is torsion.  Then $\lambda$ has either two or infinitely many simple Reeb orbits.
\end{theorem}

So, for example, assumptions (a) and (b) in Theorem~\ref{thm:hwz} can be dropped. For a different application of Theorem~\ref{thm:main}, we recall that in \cite{wh}, the second author and Taubes showed:

\begin{theorem}
\label{thm:wh}
\cite{wh}
Let $Y$ be a closed three-manifold with a nondegenerate contact form with exactly two simple Reeb orbits.  Then both orbits are elliptic and $Y$ is $S^3$ or a lens space\footnote{In \cite{wh} one just wrote that ``$Y$ is a lens space'', considering $S^3$ to be a special case of a lens space.}.  
\end{theorem}

By combining this with Theorem~\ref{thm:main}, we obtain:

\begin{corollary}
\label{cor:notlens}
Let $Y$ be a closed connected three-manifold which is not $S^3$ or a lens space. Then every nondegenerate contact form $\lambda$ on $Y$ such that $c_1(\xi)\in H^2(Y;\Z)$ is torsion has infinitely many simple Reeb orbits.
\end{corollary}

When $Y$ is $S^3$ or a lens space, we can still combine Theorems~\ref{thm:main} and \ref{thm:wh} to deduce that if a nondegenerate contact form on $Y$ has at least one hyperbolic Reeb orbit, then it has infinitely many simple Reeb orbits\footnote{In particular, if $\lambda$ is a nondegenerate contact form on a closed three-manifold $Y$, and if $\xi$ is overtwisted, then $\lambda$ has at least one positive hyperbolic simple Reeb orbit.  This follows from the fact that the ECH contact invariant of $\xi$ vanishes; see \cite[\S 1.4]{bn}.}.

\subsection{Idea of the proof of the main theorem}

The strategy of the proof of Theorem~\ref{thm:main}, inspired by \cite{hwz1}, is to use holomorphic curves to find a genus zero ``global surface of section'' for the Reeb flow; see Definition~\ref{def:gss}. If $\Sigma$ is a global surface of section, then the Reeb orbits consist of the periodic orbits of a Poincar\'e return map from $\Sigma$ to itself (which preserves the area form on $\Sigma$ given by the restriction of $d\lambda$), together with the Reeb orbits at the boundary of $\Sigma$. If $\Sigma$ has genus zero, then one can deduce the existence of either two or infinitely many simple Reeb orbits by using a theorem of Franks, asserting that an area-preserving homeomorphism of an open annulus has either zero or infinitely many periodic orbits.

In fact, we cannot always find a global surface of section. But we can find one if we assume that $\lambda$ is nondegenerate, that $c_1(\xi)$ is torsion, and that there are only finitely many simple Reeb orbits; and this is enough to prove Theorem~\ref{thm:main}.

To find a global surface of section under these hypotheses, we use embedded contact homology\footnote{In particular, both our argument and the argument of Colin-Honda \cite{ch} mentioned above use Floer homology. The methods of proof, however, are quite different: the approach in \cite{ch} involves detecting Reeb orbits directly by showing exponential growth of linearized contact homology with respect to symplectic action.} (ECH). The ECH of $(Y,\lambda)$ is the homology of a chain complex which is generated by certain finite sets of simple Reeb orbits with positive integer multiplicities, and whose differential counts certain Fredholm index one $J$-holomorphic curves in $\R\times Y$, for a suitable almost complex structure $J$ on $\R\times Y$. Most importantly for the present application, ECH is equipped with a ``$U$ map'', which is induced by a chain map which counts certain Fredholm index two $J$-holomorphic curves in $\R\times Y$. It was shown by Taubes \cite{echswf} that there is a canonical isomorphism between ECH and a version of Seiberg-Witten Floer cohomology, which identifies the $U$ map on ECH with a corresponding ``$U$ map'' on Seiberg-Witten Floer cohomology. By results of Kronheimer-Mrowka \cite{km} on the nontriviality of the latter, it then follows that there are infinitely many\footnote{More precisely, there are infinitely many different nonempty moduli spaces of holomorphic curves counted by the $U$ map.} holomorphic curves in $\R\times Y$ counted by the $U$ map on ECH. This gives us a large supply of holomorphic curves in $\R\times Y$, and we would like to show that at least one of these holomorphic curves projects to a global surface of section in $Y$.

Proposition~\ref{prop:gss} gives general criteria for a holomorphic curve $C$ in $\R\times Y$ to project to a genus zero global surface of section in $Y$. The two most nontrivial criteria to satisfy are the following: First, $C$ must have genus zero; we need this condition both for its own sake and to get an embedded surface in $Y$. Second, the component of the moduli space of holomorphic curves containing $C$ must be compact; this condition implies that these holomorphic curves fill up all of $Y$, except for the Reeb orbits at their ends. Without this condition, the moduli space component containing $C$ would only allow us to describe the dynamics on part of $Y$.  On the other hand, when the conditions in Proposition~\ref{prop:gss} are satisfied, the projections of the holomorphic curves in the same component of $C$ give a foliation with leaf space $S^1$ of the part of $Y$ away from the Reeb orbits at their ends, and the Reeb vector field is transverse to this foliation.

{\em A priori}, the holomorphic curves counted by the $U$ map need not satisfy either of these criteria.  The key new insight of our paper is that one may use the ``volume property" of ECH from \cite{vc} to control both the genus and the potential breakings of these curves.  The volume property is perhaps the deepest property of ECH; it gives a relation between the symplectic action (total period of Reeb orbits) needed to represent classes in ECH and the contact volume of $(Y,\lambda)$.

Our argument for controlling the genus through the volume property uses the ``$J_0$ index", which can be regarded as a formalism encoding the relative adjunction formula.  The $J_0$ index of a curve bounds its topological complexity.  In general, $J_0$ of a holomorphic curve depends on its relative homology class.  However, when $c_1(\xi)$ is torsion, $J_0$ of any holomorphic curve that we find using the $U$ map depends only on the asymptotics of the curve.  The idea of our argument is then to take a sequence of curves counted by $U^N$, and bound the sum of $J_0$ of  these curves.  We use the fact that there are only finitely many simple Reeb orbits to get a bound on this sum in terms of the symplectic action needed to represent a corresponding class in ECH, and we then use the volume identity to get a strong enough bound to show that most of these $N$ curves must have genus $0$ when $N$ is sufficiently large.  In the simpler situation where there are exactly two Reeb orbits, some similar arguments were used in \cite{wh} to prove Theorem~\ref{thm:wh} above; in this situation, however, the volume property was not needed.

We use the volume property again to show that for many of the genus zero curves counted by the $U$ map, the sets of Reeb orbits that they go between have a very small difference in symplectic action; see Lemma~\ref{lem:Usequence}.  By using the assumption that there are only finitely many simple Reeb orbits, and by using the fact that curves counted by the $U$ map satisfy certain constraints on their asymptotics encoded by the ``partition conditions", a combinatorial argument in \S\ref{sec:esc} finds such a curve for which there is no intermediate set of Reeb orbits along which the curve can break, so that the moduli space component is compact.

We remark that the volume property was also used in \cite{two} to prove Theorem~\ref{thm:two} above; for some additional applications of the volume property, see \cite{ai,calabi,irie}.

\subsection{Other results}
\label{sec:other}

To fully answer Question~\ref{que:2infinity}, one would like to generalize Theorem~\ref{thm:main} by dropping the assumptions that $c_1(\xi)$ is torsion and $\lambda$ is nondegenerate. We cannot currently drop the assumption that $c_1(\xi)$ is torsion for reasons alluded to above and explained more in Remark~\ref{rem:technical}; and most of the machinery we use makes extensive use of nondegeneracy. However we can still say the following about the non-torsion and possibly degenerate case:

\begin{theorem}
\label{thm:nontorsion}
Let $\lambda$ be a contact form on a closed three-manifold such that $c_1(\xi)$ is not torsion. Then:
\begin{description}
\item{(a)}
$\lambda$ has at least three simple Reeb orbits.
\item{(b)}
If $\lambda$ is nondegenerate, then $\lambda$ has at least four simple Reeb orbits.
\end{description}
\end{theorem}

The proof of Theorem~\ref{thm:nontorsion} is different and simpler than that of the main theorem, although it still uses the volume property of ECH.

Finally, in the course of the proof of Theorem~\ref{thm:main}, we obtain another result which involves refining the three-dimensional Weinstein conjecture by looking for Reeb orbits of particular types.

\begin{question}
\label{question:ph}
Let $Y$ be a closed connected three-manifold which is not $S^3$ or a lens space, and let $\lambda$ be a nondegenerate contact form on $Y$. Does $\lambda$ have a positive hyperbolic simple Reeb orbit?
\end{question}

By Theorem~\ref{thm:wh}, under the assumptions of Question~\ref{question:ph} there exists a hyperbolic simple Reeb orbit, which however might not be positive hyperbolic. We can say a bit more here:

\begin{proposition}
\label{prop:odd}
Every nondegenerate contact form on a closed three-manifold $Y$ with $b_1(Y)>0$ has a positive hyperbolic simple Reeb orbit.
\end{proposition}   

Proposition~\ref{prop:odd} is proved in \S\ref{sec:ech} as a direct corollary of the isomorphism between ECH and Seiberg-Witten Floer cohomology.

\subsection{Outline of the rest of the paper}

Section \ref{sec:ech} reviews everything that we will need to know about embedded contact homology. In particular, the $J_0$ index is reviewed in \S\ref{sec:J0}, and the volume property is reviewed in \S\ref{sec:asymp}. Section \ref{sec:gss} proves Proposition~\ref{prop:gss}, which gives general criteria for a holomorphic curve in $\R\times Y$ to project to a global surface of section for the Reeb flow in $Y$. The heart of our argument, Section~\ref{sec:curve}, uses ECH to find a holomorphic curve satisfying these criteria, assuming that $\lambda$ is nondegenerate, $c_1(\xi)$ is torsion, and there are only finitely many simple Reeb orbits. Section~\ref{sec:franks} reviews the theorem of Franks and completes the proof of Theorem~\ref{thm:main}. Finally, Section~\ref{sec:final} proves Theorem~\ref{thm:nontorsion}. The appendix clarifies the facts from Seiberg-Witten theory that are needed to yield infinitely many holomorphic curves counted by the $U$ map on ECH.

\section{Embedded contact homology preliminaries}
\label{sec:ech}

Let $Y$ be a closed connected three-manifold, let $\lambda$ be a nondegenerate contact form on $Y$, and let $\Gamma\in H_1(Y)$. We now review the definition of the embedded contact homology $ECH_*(Y,\lambda,\Gamma)$, and the facts about this that we will need. More details about ECH can be found in the lecture notes \cite{bn}.

Roughly speaking, $ECH_*(Y,\lambda,\Gamma)$ is built from finite sets of simple Reeb orbits with multiplicities with total homology class $\Gamma$. For the proof of Theorem~\ref{thm:main}, we will just need to consider the case $\Gamma=0$; however we will need to work with other classes $\Gamma$ in the proof of Theorem~\ref{thm:nontorsion}. 

\subsection{Holomorphic curves and currents}
\label{sec:currents}

We say that an almost complex structure $J$ on $\R\times Y$ is {\bf $\lambda$-compatible\/} if $J$ is $\R$-invariant; $J(\partial_s)=R$ where $s$ denotes the $\R$ coordinate on $\R\times Y$ and $R$ denotes the Reeb vector field as usual; and $J(\xi)=\xi$, rotating $\xi$ positively with respect to $d\lambda$. Fix a $\lambda$-compatible $J$.

We consider $J$-holomorphic curves of the form $u:(\Sigma,j)\to(\R\times Y,J)$ where the domain $(\Sigma,j)$ is a punctured compact Riemann surface. Here the domain $\Sigma$ is not necessarily connected, and we say that $u$ is {\bf irreducible\/} if $\Sigma$ is connected. If $\gamma$ is a (not necessarily simple) Reeb orbit, a {\bf positive end\/} of $u$ at $\gamma$ is a puncture near which $u$ is asymptotic to $\R\times\gamma$ as $s\to\infty$, and a {\bf negative end\/} of $u$ at $\gamma$ is a puncture near which $u$ is asymptotic to $\R\times \gamma$ as $s\to -\infty$; see \cite[\S3.1]{bn} for more details. We assume that each puncture is a positive end or a negative end as above. We mod out by the usual equivalence relation on holomorphic curves, namely composition with biholomorphic maps between domains. Under this equivalence relation, if $u$ is somewhere injective, then $u$ is determined by its image $C=u(\Sigma)$, and in this case we often abuse notation to denote $u$ by its image $C$.

If $u$ is a $J$-holomorphic curve as above, its {\bf Fredholm index\/} is defined by
\begin{equation}
\label{eqn:Fredholmindex}
\op{ind}(u) = -\chi(\Sigma) + 2c_\tau(u) + \op{CZ}_\tau^{\op{ind}}(u).
\end{equation}
Here $\tau$ is a symplectic trivialization of the contact structure $\xi$ over the Reeb orbits at which $u$ has ends. The term $c_\tau(u)$ denotes the relative first Chern class of $u^*\xi$ with respect to $\tau$, see \cite[\S3.2]{bn}. Finally, suppose that $u$ has $k$ positive ends at (not necessarily simple) Reeb orbits $\gamma_1^+,\ldots,\gamma_k^+$, and $l$ negative ends at Reeb orbits $\gamma_1^-,\ldots,\gamma_j^-$. Then the last term is defined by
\[
\op{CZ}_\tau^{\op{ind}}(u) = \sum_{i=1}^k\op{CZ}_\tau(\gamma_i^+) - \sum_{j=1}^l\op{CZ}_\tau(\gamma_j^-).
\]
Here if $\gamma$ is a Reeb orbit and $\tau$ is a trivialization of $\gamma^*\xi$, then $\op{CZ}_\tau(\gamma)$ denotes the Conley-Zehnder index of $\gamma$ with respect to $\tau$. In our three-dimensional situation, this is given by
\begin{equation}
\label{eqn:CZequation}
\op{CZ}_\tau(\gamma) = \floor{\theta} + \ceil{\theta}
\end{equation}
where $\theta$ denotes the rotation number with respect to $\tau$ of the linearized Reeb flow along $\gamma$; see \cite[\S3.2]{bn}. The Fredholm index does not depend on the choice of trivialization $\tau$. The significance of the Fredholm index is that if $J$ is generic and if $u$ is irreducible and somewhere injective, then the moduli space of $J$-holomorphic curves near $u$ is a manifold of dimension $\op{ind}(u)$.

Sometimes we wish to mod out by a further equivalence relation, declaring two $J$-holomorphic curves to be equivalent if they represent the same current in $\R\times Y$. In this case we refer to an equivalence class as a {\bf $J$-holomorphic current\/}. A $J$-holomorphic current is described by a finite set of pairs $\mc{C}=\{(C_k,d_k)\}$, where the $C_k$ are distinct irreducible somewhere injective $J$-holomorphic curves as above, which we refer to as the ``components'' of $\mc{C}$, and the $d_k$ are positive integers which we refer to as the ``multiplicities'' of these components.

\subsection{Definition of embedded contact homology}
\label{sec:echdef}

We are now ready to define the embedded contact homology $ECH(Y,\lambda,\Gamma)$. This is the homology of a chain complex $ECC_*(Y,\lambda,\Gamma)$ over $\Z/2$ defined as follows\footnote{It is also possible to define ECH with $\Z$ coefficients, as explained in \cite[\S9]{obg2}, but this is not necessary for the applications so far.}. An {\bf orbit set\/} is a finite set of pairs $\alpha = \lbrace (\alpha_i,m_i) \rbrace$ where the $\alpha_i$ are distinct simple Reeb orbits, and the $m_i$ are positive integers.  The orbit set $\alpha$ is {\bf admissible} if $m_i=1$ whenever $\alpha_i$ is hyperbolic. The total homology class of the orbit set $\alpha$ is defined by
\[
[\alpha] \eqdef \sum_i m_i [\alpha_i] \in H_1(Y).
\]
The chain complex $ECC_*(Y,\lambda,\Gamma)$ is now freely generated over $\Z/2$ by admissible orbit sets $\alpha$ with total homology class $[\alpha]=\Gamma$. We sometimes write an orbit set $\alpha = \lbrace (\alpha_i,m_i) \rbrace$ as a commutative product $\alpha = \prod_i \alpha_i^{m_i}$, and we usually refer to an admissible orbit set as an {\bf ECH generator}.

The differential on the chain complex $ECC_*(Y,\lambda,\Gamma)$ depends on the additional choice of a generic $\lambda$-compatible almost complex structure $J$. If $\alpha=\{(\alpha_i,m_i)\}$ and $\beta=\{(\beta_j,n_j)\}$ are (not necessarily admissible) orbit sets with $[\alpha]=[\beta]\in H_1(Y)$, let $\M^J(\alpha,\beta)$ denote the set of $J$-holomorphic currents as in \S\ref{sec:currents} with positive ends at covers of $\alpha_i$ with total covering multiplicity $m_i$, negative ends at covers of $\beta_j$ with total covering multiplicity $n_j$, and no other punctures.

Continuing to assume that $[\alpha]=[\beta]$, let $H_2(Y,\alpha,\beta)$ denote the set of 2-chains $Z$ in $Y$ with $\partial Z = \sum_im_i\alpha_i - \sum_jn_j\beta_j$, modulo boundaries of 3-chains. The set $H_2(Y,\alpha,\beta)$ is an affine space over $H_2(Y)$, and each current $\mc{C}\in \M^J(\alpha,\beta)$ determines a relative homology class $[\mc{C}]\in H_2(Y,\alpha,\beta)$.

Given $Z\in H_2(Y,\alpha,\beta)$, we define the {\bf ECH index\/}
\begin{equation}
\label{eqn:defI}
I(\alpha,\beta,Z) = c_\tau(Z) + Q_\tau(Z) +  \sum_i\sum_{k=1}^{m_i}\op{CZ}_\tau(\alpha_i^k) - \sum_j\sum_{l=1}^{n_j}\op{CZ}_\tau(\beta_j^l).
\end{equation}
Here $\tau$ is a trivialization of $\xi$ over the Reeb orbits $\alpha_i$ and $\beta_j$, and $c_\tau$ denotes the relative first Chern class as before. The integer $Q_\tau(Z)$ is the ``relative self-intersection number'' reviewed in \cite[\S3.3]{bn}.
In the Conley-Zehnder index terms, if $\gamma$ is a Reeb orbit and $k$ is a positive integer, then $\gamma^k$ denotes the Reeb orbit which is a $k$-fold cover of $\gamma$.

The ECH index does not depend on the choice of trivialization $\tau$. However the ECH index sometimes does depend on the relative homology class $Z$. Namely, if $[\alpha]=[\beta]=\Gamma\in H_1(Y)$, and if $Z,Z'\in H_2(Y,\alpha,\beta)$, then the difference $Z-Z'$ is an element of $H_2(Y)$, and we have the ``index ambiguity formula''
\begin{equation}
\label{eqn:indexambiguity}
I(\alpha,\beta,Z) - I(\alpha,\beta,Z') = \langle c_1(\xi) + 2 \op{PD}(\Gamma), Z - Z' \rangle,
\end{equation}
see \cite[Eq. 3.6]{bn}.  The ECH index is also additive in the following sense: If $\delta$ is another orbit set with $[\alpha]=[\beta]=[\delta]$, if $Z\in H_2(Y,\alpha,\beta)$, and if $W\in H_2(Y,\beta,\delta)$, then $Z+W\in H_2(Y,\alpha,\delta)$ is defined and
\begin{equation}
\label{eqn:Iadditive}
I(\alpha,\beta,Z) + I(\beta,\delta,W) = I(\alpha,\delta,Z+W),
\end{equation}
see \cite[\S3.4]{bn}.

Given a current $\mc{C}\in\M^J(\alpha,\beta)$, we define its ECH index $I(\mc{C}) = I(\alpha,\beta,[\mc{C}])$. We also write $c_\tau(\mc{C})=c_\tau([\mc{C}])$ and $Q_\tau(\mc{C}) = Q_\tau([\mc{C}])$. If $k$ is an integer, we define $\M^J_k(\alpha,\beta)$ to be the set of $\mc{C}\in\M^J(\alpha,\beta)$ with ECH index $I(\mc{C})=k$.

The significance of the ECH index is that it bounds the Fredhom index via thefollowing index inequality, explained in \cite[\S3.4]{bn}: If $C\in\M^J(\alpha,\beta)$ has no multiply covered components, then
\begin{equation}
\label{eqn:ii}
\op{ind}(C) \le I(C) - 2\delta(C).
\end{equation}
Here $\delta(C)$ is a count of the singularities of $C$ with positive integer multiplicities.

The index inequality \eqref{eqn:ii} leads to the following classification of holomorphic currents with low ECH index. Below, define a {\bf trivial cylinder\/} to be a cylinder $\R\times\gamma\subset\R\times Y$ where $\gamma$ is a simple Reeb orbit. A trivial cylinder is an embedded $J$-holomorphic curve for any $\lambda$-compatible $J$.

\begin{proposition}
\label{prop:lowI}
\cite[Prop.\ 3.7]{bn}
Let $J$ be a generic $\lambda$-compatible almost complex structure. Let $\alpha$ and $\beta$ be orbit sets with $[\alpha]=[\beta]$ and let $\mc{C}\in \mc{M}^J(\alpha,\beta)$. Then:
\begin{description}
\item{(0)} $I(\mc{C})\ge 0$, with equality if and only if each component of $\mc{C}$ is a trivial cylinder.
\item{(1)} If $I(\mc{C})=1$, then $\mc{C}=\mc{C}_0\sqcup C_1$, where $I(\mc{C}_0)=0$, and $C_1$ is embedded and does not include any trivial cylinders and has $\op{ind}(C_1)=I(C_1)=1$.
\item{(2)} If $\alpha$ and $\beta$ are admissible and $I(\mc{C})=2$, then $\mc{C}=\mc{C}_0\sqcup C_1$, where $I(\mc{C}_0)=0$, and $C_1$ is embedded and does not include any trivial cylinders and has $\op{ind}(C_1)=I(C_1)=2$.
\end{description}
\end{proposition}

In particular, it follows from Proposition~\ref{prop:lowI}(1) that $\mc{M}^J_1(\alpha,\beta)/\R$ is a discrete set, where $\R$ acts on $\mc{M}^J(\alpha,\beta)$ by translation of the $\R$ factor in $\R\times Y$.
The differential on the ECH chain complex $ECC(Y,\lambda,\Gamma)$ is now defined as follows: Choose a generic $\lambda$-compatible $J$. Given an admissible orbit set $\alpha$ with $[\alpha]=\Gamma$, define
\[
\partial_J\alpha=\sum_{\beta}\#(\mc{M}^J_1(\alpha,\beta)/\R)\beta.
\]
Here the sum is over admissible orbit sets $\beta$ with $[\beta]=\Gamma$, and `$\#$' denotes the mod 2 count. It is shown in \cite[\S5.3]{bn} that $\partial_J$ is well-defined, and in \cite[Thm.\ 7.20]{obg1} that $\partial_J^2=0$. The embedded contact homology $ECH_*(Y,\lambda,\Gamma)$ is now defined to be the homology of the chain complex $(ECC_*(Y,\lambda,\Gamma),\partial_J)$. Although the differential $\partial_J$ may depend on $J$, the homology of the chain complex does not; see Theorem~\ref{thm:taubes} below.

The ECH index induces a relative $\Z/d$ grading on the chain complex $ECC_*(Y,\lambda,\Gamma)$, where $d$ denotes the divisibility of the cohomology class $c_1(\xi)+2\op{PD}(\Gamma)$ in $H^2(Y;\Z)$ mod torsion. If $\alpha$ and $\beta$ are generators with $[\alpha]=[\beta]=\Gamma$, the grading difference between $\alpha$ and $\beta$ is defined by
\[
I(\alpha,\beta) = I(\alpha,\beta,Z)
\]
where $Z\in H_2(Y,\alpha,\beta)$. The relative grading does not depend on $Z$ as a result of the index ambiguity formula \eqref{eqn:indexambiguity}. By definition, the differential $\partial_J$ decreases the relative grading by $1$.

There is also an absolute $\Z/2$ grading $I_2$ on the chain complex defined as follows: If $\alpha=\{(\alpha_i,m_i)\}$ is an admissible orbit set, then $I_2(\alpha)$ is the mod 2 count of orbits $\alpha_i$ that are positive hyperbolic. This is compatible with the relative grading $I$ in the sense that if $[\alpha]=[\beta]$, then
\begin{equation}
\label{eqn:parity}
I(\alpha,\beta) \equiv I_2(\alpha) - I_2(\beta) \mod 2,
\end{equation}
see \cite[Prop.\ 1.6(c)]{pfh2}.

\subsection{The $U$ map}
\label{sec:Ucurves}

Embedded contact homology has various additional structures on it. One such structure that will play a crucial role in this paper is the {\bf {\em U\/} map}, a degree $-2$ map
\begin{equation}
\label{eqn:Udef}
U: ECH_*(Y,\lambda,\Gamma) \longrightarrow ECH_{*-2}(Y,\lambda,\Gamma).
\end{equation}
To define this, choose a base point $z\in Y$ which is not on the image of any Reeb orbit, and let $J$ be a generic $\lambda$-compatible almost complex structure. One then defines a map
\[
U_{J,z}: ECC_*(Y,\lambda,\Gamma) \longrightarrow ECH_{*-2}(Y,\lambda,\Gamma)
\]
as follows: if $\alpha$ and $\beta$ are ECH generators, then the coefficient $\langle U_{J,z}\alpha,\beta\rangle$ is the mod 2 count of $J$-holomorphic currents in $\M^J_2(\alpha,\beta)$ that pass through the point $(0,z)\in\R\times Y$.

As explained in \cite[\S2.5]{wh}, the map $U_{J,z}$ is a chain map, and we define the $U$ map \eqref{eqn:Udef} to be the induced map on homology. Our assumption that $Y$ is connected implies that $U$ does not depend on the choice of $z$. The $U$ map does not depend on $J$ either by Theorem~\ref{thm:taubes} below.

If $\alpha$ and $\beta$ are ECH generators and if $\mc{C}\in\M^J_2(\alpha,\beta)$ is a $J$-holomorphic current counted by the chain map $U_{J,z}$, then we refer to $\mc{C}$ as a {\bf {\em U\/}-curve\/}. By Proposition~\ref{prop:lowI}, a $U$-curve has the form
\[
\mc{C} = \mc{C}_0 \sqcup C_1
\]
where $\mc{C}_0$ is a union of trivial cylinders with multiplicities, and $C_1$ is embedded and satisfies $\op{ind}(C_1)=I(C_1)=2$. Moreover, $C_1$ is irreducible by \cite[Lem.\ 2.6(b)]{wh}.

\subsection{The isomorphism with Seiberg-Witten theory} 
\label{sec:Taubes}

{\em A priori\/}, the embedded contact homology $ECH_*(Y,\lambda,\Gamma)$ could depend on the choice of generic $\lambda$-compatible $J$, so strictly speaking we should denote it by $ECH_*(Y,\lambda,\Gamma,J)$. The $U$ map could also depend on $J$. In fact, these depend only on $Y$, $\Gamma$, and the contact structure $\xi$, as a result of the following theorem of Taubes:

\begin{theorem}
\label{thm:taubes}
\cite{echswf}
Let $Y$ be a closed connected three-manifold with a nondegenerate contact form $\lambda$, and let $\Gamma\in H_1(Y)$.
Then for any generic $\lambda$-compatible almost complex structure $J$, there is a canonical isomorphism of relatively graded $\Z/2$-modules\footnote{One can also obtain an isomorphism with $\Z$ coefficients, see \cite{e3}.}
\begin{equation}
\label{eqn:taubes}
ECH_*(Y,\lambda,\Gamma,J) \stackrel{\simeq}{\longrightarrow} \widehat{HM}^{-*}(Y,\frak{s}_{\xi}+\op{PD}(\Gamma); \Z/2)
\end{equation}
which preserves the $U$ maps on both sides.
\end{theorem}

Here $\widehat{HM}^{*}(Y,\frak{s}; \Z/2)$ is a version of Seiberg-Witten Floer cohomology with $\Z/2$ coefficients defined by Kronheimer-Mrowka \cite{km}, which depends on a closed oriented connected three-manifold $Y$ together with a spin-c structure $\frak{s}$. This has a relative $\Z/d$ grading, where $d$ denotes the divisibility of $c_1(\frak{s})$ in $H^2(Y;\Z)$ mod torsion.  In the theorem, $\frak{s}_{\xi}$ denotes the spin-c structure determined by the oriented $2$-plane field $\xi$, see e.g.\ \cite[\S2.8]{bn}. We have
\begin{equation}
\label{eqn:spinc}
c_1(\frak{s}_\xi+\op{PD}(\Gamma)) = c_1(\xi) + 2\op{PD}(\Gamma),
\end{equation}
so that both sides of \eqref{eqn:taubes} have the same type of relative grading. Also, the group $\widehat{HM}^*(Y,\frak{s}; \Z/2)$ is equipped with a canonical degree $2$ map which is also denoted by $U$.

In addition to implying topological invariance of ECH, Theorem~\ref{thm:taubes}, combined with known results about Seiberg-Witten Floer cohomology, implies nontriviality results for ECH. In particular, the following proposition provides an abundant supply of $U$-curves which will be needed in the proof of the main theorem. Below, define a {\bf $U$-sequence\/} to be an infinite sequence $\{\sigma_k\}_{k\ge 1}$ of nonzero homogeneous classes in ECH such that $U\sigma_{k+1}=\sigma_k$ for each $k\ge 1$. Also, use the canonical $\Z/2$ grading $I_2$ on ECH to decompose
\[
ECH_*(Y,\lambda,\Gamma) = ECH_{\op{even}}(Y,\lambda,\Gamma) \oplus ECH_{\op{odd}}(Y,\lambda,\Gamma).
\]

\begin{proposition}
\label{prop:Useq}
Let $Y$ be a closed connected three-manifold with a nondegenerate contact form $\lambda$, and let $\Gamma\in H_1(Y)$ such that $c_1(\xi) + 2\op{PD}(\Gamma)\in H^2(Y;\Z)$ is torsion. Then:
\begin{description}
\item{(a)} There exists a $U$-sequence in $ECH_*(Y,\lambda,\Gamma)$.
\item{(b)} If $b_1(Y)>0$, then there exist $U$-sequences in both $ECH_{\op{even}}(Y,\lambda,\Gamma)$ and $ECH_{\op{odd}}(Y,\lambda,\Gamma)$.
\end{description}
\end{proposition}

The proof of Proposition~\ref{prop:Useq} is given in Appendix~\ref{app:swfacts}. We can now use this proposition to give:

\begin{proof}[Proof of Proposition~\ref{prop:odd}.] Without loss of generality, $Y$ is connected.
Choose $\Gamma$ such that $c_1(\xi) + 2 \op{PD}(\Gamma)$ is torsion (such a $\Gamma$ always exists). By Proposition~\ref{prop:Useq}(b), there exists a $U$-sequence in $ECH_{\op{odd}}(Y,\lambda,\Gamma)$. In particular, $ECH_{\op{odd}}(Y,\lambda,\Gamma)$ is nontrivial, so there exists an ECH generator $\alpha=\{(\alpha_i,m_i)\}$ with $[\alpha]=\Gamma$ and $I_2(\alpha)=1$. From the definition of $I_2$, it follows that at least one of the simple Reeb orbits $\alpha_i$ is positive hyperbolic.
\end{proof}

\subsection{Partition conditions}
\label{sec:partition}

The nontrivial components of $U$-curves satisfy specific constraints on their asymptotics which we will need to take into account. To state these, let $\gamma$ be a simple Reeb orbit and let $m$ be a positive integer. 
We now define two partitions of $m$, the ``positive partition'' $p^+_\gamma(m)$ and the ``negative partition'' $p^-_\gamma(m)$, as follows. Let $\theta\in\R$ be the rotation number of $\gamma$ as in \eqref{eqn:CZequation} with respect to some trivialization $\tau$ of $\xi|_\gamma$. We then define $p^\pm_\gamma(m)=p^\pm_\theta(m)$, where $p^\pm_\theta(m)$ is defined as follows.

Let $\Lambda_\theta^+(m)$ denote the maximal polygonal path in the plane from $(0,0)$ to $(m,\floor{m\theta})$ with vertices at lattice points which is the graph of a concave function and which does not go above the line $y=\theta x$. Then $p_\theta^+(m)$ consists of the horizontal components of the segments of the path $\Lambda_\theta^+(m)$ connecting consecutive lattice points. The partition $p_\theta^-(m)$ is obtained likewise from $\Lambda_\theta^-(m)$, which is the minimal polygonal path in the plane from $(0,0)$ to $(m,\ceil{m\theta})$ with vertices at lattice points which is the graph of a convex function and which does not go below the line $y=\theta x$. Equivalently,
\[
p_\theta^-(m) = p_{-\theta}^+(m).
\]
Note that the partition $p_\theta^\pm(m)$ depends only on the congruence class of $\theta\in\R/\Z$, which does not depend on the choice of trivialization $\tau$.

For example, if $\gamma$ is positive hyperbolic, then $\theta$ is an integer, and it follows that
\begin{equation}
\label{eqn:partitionh+}
p^+_{\gamma}(m)=p^-_{\gamma}(m)=(1,\ldots,1).
\end{equation}
If $\gamma$ is negative hyperbolic, then $\theta\equiv 1/2\mod\Z$, and it follows that
\begin{equation}
\label{eqn:partitionh-}
p_\gamma^+(m) = p_\gamma^-(m) = \left\{\begin{array}{cl} (2,\ldots,2), & \mbox{$m$ even,}\\ (2,\ldots,2,1), & \mbox{$m$ odd.}
\end{array}\right.
\end{equation}
If $\gamma$ is elliptic, then our usual assumption that all Reeb orbits are nondegenerate implies that $\theta$ is irrational, and it then turns out that $p_\theta^+(m)$ and $p_\theta^-(m)$ are disjoint whenever $m>1$, see \cite[Ex.\ 3.13]{bn}.

The significance of these partitions is as follows. Let $\alpha=\{(\alpha_i,m_i)\}$ and $\beta=\{(\beta_j,n_j)\}$ be orbit sets, and suppose that $C\in\M^J(\alpha,\beta)$ has no multiply covered components. Then for each $i$, the curve $C$ has ends at covers of $\alpha_i$ with total multiplicity $m_i$, and these multiplicities determine a partition of $m_i$, which we denote by $p_{\alpha_i}^+(C)$. Likewise, for each $j$, the asymptotics of the negative ends of $C$ at covers of $\beta_j$ determine a partition of $n_j$, which we denote by $p_{\beta_j}^-(C)$. We then have:

\begin{proposition}
\label{prop:partitions}
\cite[\S3.9]{bn}
Suppose that $C\in\M^J(\alpha,\beta)$ has no multiply covered components and that equality holds in the index inequality \eqref{eqn:ii}. Then for each $i$ we have $p_{\alpha_i}^+(C)=p_{\alpha_i}^+(m_i)$, and for each $j$ we have $p_{\beta_j}^-(C) = p_{\beta_j}^-(n_j)$.
\end{proposition}

In particular, if $\mc{C}=\mc{C}_0\sqcup C_1$ is a $U$-curve, then by Proposition~\ref{prop:lowI}(2), we know that Proposition~\ref{prop:partitions} is applicable to the nontrivial component $C_1$.

We note one simple fact about the partitions which will be needed later.   
Let $\gamma$ be a simple Reeb orbit and let $m$ be a positive integer. Recall that without any choice of trivialization, $\gamma$ has a well-defined rotation number $\theta\in\R/\Z$. In particular, 
$\floor{m\theta}$ and $\ceil{m\theta}$ are well-defined elements of $\Z/m\Z$.

\begin{lemma}
\label{lem:relprime}
Let $\gamma$ be a simple Reeb orbit with rotation number $\theta\in\R/\Z$, and let $m$ be a positive integer.
\begin{description}
\item{(a)}
If $p_\gamma^+(m)=(m)$, then $\op{gcd}(m,\floor{m\theta})=1$.
\item{(b)}
If $p_\gamma^-(m)=(m)$, then $\op{gcd}(m,\ceil{m\theta})=1$.
\end{description}
\end{lemma}

\begin{proof}
Let $\theta\in\R$ be the rotation number of $\gamma$ for some trivialization of $\xi|_\gamma$. If $p_\gamma^+(m)=(m)$, then this means that the path $\Lambda_\theta^+(m)$ consists of the single edge from $(0,0)$ to $(m,\floor{m\theta})$, and this edge has no lattice point in its interior. It follows that $\floor{m\theta}$ is relatively prime to $m$. This proves (a), and (b) is proved by a symmetric argument.
\end{proof}

\subsection{The $J_0$ index}
\label{sec:J0}

We now recall a variant of the ECH index which is useful for bounding the topological complexity of holomorphic curves. 

Let $\alpha=\{(\alpha_i,m_i)\}$ and $\beta=\{(\beta_j,n_j)\}$ be orbit sets with $[\alpha]=[\beta]$, and let $Z\in H_2(Y,\alpha,\beta)$. We then define the ``$J_0$ index''
\begin{equation}
\label{eqn:defJ0}
J_0(\alpha,\beta,Z) = -c_\tau(Z) + Q_\tau(Z) + 
\sum_i\sum_{k=1}^{m_i-1}\op{CZ}_\tau(\alpha_i^k) - \sum_j\sum_{l=1}^{n_j-1}\op{CZ}_\tau(\beta_j^l).
\end{equation}
Here $\tau$, $c_\tau$, and $Q_\tau$ are defined as in \eqref{eqn:defI}.

Like the ECH index $I$, the $J_0$ index \eqref{eqn:defJ0} does not depend on the choice of trivialization $\tau$. However $J_0$ depends on $Z$ in the following way, similarly to the index ambiguity formula \eqref{eqn:indexambiguity}. If $\alpha$ and $\beta$ are orbit sets with $[\alpha]=[\beta]=\Gamma$, and if $Z,Z'\in H_2(Y,\alpha,\beta)$, then the ambiguity in the relative first Chern class is given by
\begin{equation}
\label{eqn:ctauambiguity}
c_\tau(Z) - c_\tau(Z') = \langle c_1(\xi),Z-Z' \rangle.
\end{equation} 
This follows from the definition of the relative first Chern class in \cite[\S3.2]{bn}.  By \eqref{eqn:ctauambiguity}, together with \eqref{eqn:defI}, \eqref{eqn:indexambiguity}, \eqref{eqn:defJ0}, we obtain
\begin{equation}
\label{eqn:J0ambiguity}
J_0(\alpha,\beta,Z) - J_0(\alpha,\beta,Z') = \langle -c_1(\xi) + 2 \op{PD}(\Gamma), Z - Z' \rangle.
\end{equation}

If $\mc{C}\in \M^J(\alpha,\beta)$, we write $J_0(\mc{C}) = J_0(\alpha,\beta,[\mc{C}])$. If $\mc{C}$ is a $U$-curve, then the integer $J_0(\mc{C})$ gives the following bound on the topological complexity of $\mc{C}$. Write $\alpha=\{(\alpha_i,m_i)\}$ and $\beta=\{(\beta_j,n_j)\}$ and $\mc{C}=\mc{C}_0\sqcup C_1$ as usual. Let $n_i^+$ denote the number of positive ends of $C_1$ at covers of $\alpha_i$, plus $1$ if $\mc{C}_0$ includes a cover of $\R\times\alpha_i$. Let $n_j^-$ denote the number of negative ends of $C_1$ at covers of $\beta_j$, plus $1$ if $\mc{C}_0$ includes a cover of $\R\times\beta_j$.

\begin{proposition}
\label{prop:obscure}
\cite[Lem.\ 3.5]{wh},
\cite[Prop.\ 5.8]{bn}
Let $(Y,\lambda)$ be a nondegenerate contact three-manifold and let $J$ be a generic $\lambda$-compatible almost complex structure.
Let $\alpha=\{(\alpha_i,m_i)\}$ and $\beta=\{(\beta_j,n_j)\}$ be ECH generators and let $\mc{C}=\mc{C}_0\sqcup C_1\in\M^J(\alpha,\beta)$ be a $U$-curve. Then
\begin{equation}
\label{eqn:obscure}
J_0(\mc{C}) = -\chi(C_1) + \sum_i(n_i^+-1) + \sum_j(n_j^--1).
\end{equation}
\end{proposition}

\subsection{Spectral invariants and the volume property}
\label{sec:asymp}

It follows from the isomorphism \eqref{eqn:taubes} that the embedded contact homology $ECH_*(Y,\lambda,\Gamma)$ is a topological invariant.  However $ECH$ can be used to extract finer information in the form of real numbers depending on the contact form, using a filtration on the $ECH$ chain complex by the symplectic action.

If $\alpha=\{(\alpha_i,m_i)\}$ is an orbit set, its {\bf symplectic action\/} is defined by
\[
\mathcal{A}(\alpha) = \sum_i m_i \int_{\alpha_i} \lambda.
\]
The ECH differential decreases the symplectic action, i.e.\ if the coefficient $\langle\partial_J\alpha,\beta\rangle\neq 0$ then $\mc{A}(\alpha)>\mc{A}(\beta)$, see \cite[\S1.4]{bn}.
Consequently, for any $L\in \R$ we can define the {\bf filtered ECH\/}
\begin{equation}
\label{eqn:ECHL}
ECH^L_*(Y,\lambda,\Gamma)
\end{equation}
to be the homology of the subcomplex of $ECC_*(Y,\lambda,\Gamma)$ generated by orbit sets $\alpha$ with $\mathcal{A}(\alpha) < L$. There is a natural map
\begin{equation}
\label{eqn:imathL}
\imath_L: ECH^L_*(Y,\lambda,\Gamma) \longrightarrow ECH_*(Y,\lambda,\Gamma)
\end{equation}
induced by inclusion of chain complexes. It is shown in \cite[Thm.\ 1.3]{cc2} that the filtered ECH \eqref{eqn:ECHL} and the inclusion-induced map \eqref{eqn:imathL} do not depend on the choice of almost complex structure $J$.

We can now define, for each nonzero class $\sigma\in ECH_*(Y,\lambda,\Gamma)$, a ``spectral invariant''
\[
c_{\sigma}(Y,\lambda) = \op{inf} \left\{ L \;\big|\; \sigma \in \op{Im}\left(\imath_L\right)\right\}.
\]
Equivalently, $c_\sigma(Y,\lambda)$ is the smallest real number $L$ such that the class $\sigma$ can be represented by a cycle in the chain complex $ECC_*(Y,\lambda,\Gamma)$ which is a sum of ECH generators each with action $\le L$. In particular,
\begin{equation}
\label{eqn:actionrepresentation}
c_{\sigma}(Y,\lambda)=\mathcal{A}(\alpha)
\end{equation}
for some orbit set $\alpha$ which is a generator of the chain complex $ECC_*(Y,\lambda,\Gamma)$. Another useful property is that if $U\sigma\neq 0$ then
\begin{equation}
\label{eqn:actionU}
c_{U\sigma}(Y,\lambda) < c_{\sigma}(Y,\lambda).
\end{equation} 
This holds because the chain map $U_{J,z}$, like the differential, decreases symplectic action.

We are assuming above that the contact form $\lambda$ is nondegenerate. In fact, the spectral numbers $c_\sigma$ are $C^0$-continuous functions of the contact form, so one can extend them to degenerate contact forms by taking limits, see \cite[\S4.1]{qech}. When $\lambda$ is degenerate, we make sense of the `$\sigma$' in $c_\sigma$ by using the topological invariance in \eqref{eqn:taubes} to identify $ECH_*(Y,\lambda,\Gamma)$ with $ECH_*(Y,\lambda',\Gamma)$ where $\lambda'$ is a nondegenerate contact form with the same contact structure as $\lambda$. For degenerate contact forms, property \eqref{eqn:actionrepresentation} still holds, where $\alpha$ is some orbit set with $[\alpha]=\Gamma$. Property \eqref{eqn:actionU} holds in the degenerate case under the additional assumption that there are only finitely many simple Reeb orbits, see \cite[Lem.\ 3.1]{two}.

A deeper property of the spectral numbers $c_\sigma$, which will play a key role in the proof of the main theorem, is the following relation between their asymptotics and the contact volume
\[
\op{vol}(Y,\lambda) = \int_Y \lambda \wedge d\lambda.
\]
Recall from \S\ref{sec:echdef} that if $c_1(\xi) + 2 \op{PD}(\Gamma)\in H^2(Y;\Z)$ is torsion, then $ECH_*(Y,\lambda,\Gamma)$ has a relative $\Z$-grading.

\begin{theorem}\cite[Thm. 1.3]{vc}
\label{thm:vc}
Let $Y$ be a closed connected three-manifold with a contact form $\lambda$, let $\Gamma\in H_1(Y)$ with $c_1(\xi) + 2 \op{PD}(\Gamma)$ torsion, and let $I$ be any  refinement of the relative $\mathbb{Z}$-grading on $ECH_*(Y,\lambda,\Gamma)$ to an absolute $\Z$-grading.  Then for any sequence of nonzero homogeneous classes $\{\sigma_k\}_{k\ge 1}$ in $ECH_*(Y,\lambda,\Gamma)$ with $\lim_{k\to\infty}I(\sigma_k)=+\infty$, we have
\begin{equation}
\label{eqn:volumeformula}
\lim_{k \to \infty} \frac{c_{\sigma_k}(Y,\lambda)^2}{I(\sigma_k)} = \op{vol}(Y,\lambda).
\end{equation}
\end{theorem}

In particular, if $\{\sigma_k\}_{k\ge 1}$ is a U-sequence, then $I(\sigma_k)=2k+a$ for some constant $a$, so \eqref{eqn:volumeformula} implies that
\begin{equation}
\label{eqn:Useqasymptotics}
\lim_{k \to \infty} \frac{c_{\sigma_k}(Y,\lambda)^2}{k} = 2\op{vol}(Y,\lambda).
\end{equation}

\section{Criteria for a global surface of section}
\label{sec:gss}

The goal of this section is to prove Proposition~\ref{prop:gss} below, which gives criteria under which a holomorphic curve gives rise to a ``global surface of section'' for the Reeb flow. For related statements and proofs, see \cite[Prop.\ 5.1]{hwz1} and \cite[Lem.\ 6.9]{hry}.

\begin{definition}
\label{def:gss}
Let $(Y,\lambda)$ be a contact three-manifold. A {\bf global surface of section\/} for the Reeb flow is an embedded open surface $\Sigma\subset Y$ such that:
\begin{itemize}
	\item The Reeb vector field $R$ is transverse to $\Sigma$.
	\item There is a compact surface with boundary, $\overline{\Sigma}$, such that $\op{int}(\overline{\Sigma})=\Sigma$, and the inclusion $\Sigma\to Y$ extends to a continuous map $g:\overline{\Sigma}\to Y$ such that each boundary circle of $\overline{\Sigma}$ is mapped to the image of a Reeb orbit.
	\item For every $y\in Y\setminus g\left(\partial\overline{\Sigma}\right)$, the Reeb trajectory starting at $y$ intersects $\Sigma$ in both forward and backward time.
\end{itemize}
\end{definition}

\noindent
We will use the following notation. Suppose that $(Y,\lambda)$ is a nondegenerate contact three-manifold and $J$ is a $\lambda$-compatible almost complex structure on $\R\times Y$. If $u$ is a $J$-holomorphic curve in $\R\times Y$ as in \S\ref{sec:currents}, let $g(u)$ denote the genus of the domain of $u$, and let $h_+(u)$ denote the number of ends of $u$ at positive hyperbolic Reeb orbits (including even degree covers of negative hyperbolic orbits). Let $\M^J_u$ denote the component of the moduli space of $J$-holomorphic curves in $\mathbb{R} \times Y$ that contains $u$. Let $\pi_Y:\R\times Y\to Y$ denote the projection.

\begin{proposition}
\label{prop:gss}
Let $(Y,\lambda)$ be a nondegenerate contact three-manifold, and let $J$ be a $\lambda$-compatible almost complex structure on $\R\times Y$. Let $C$ be an irreducible $J$-holomorphic curve in $\R\times Y$ such that:
\begin{description}
	\item{(i)} Every $C'\in\M^J_C$ is embedded\footnote{With more work, one can weaken hypothesis (i) to just assume that $C$ is embedded in $\R\times Y$. However we will not need to do this.} in $\R\times Y$.
	\item{(ii)} $g(C)=h_+(C)=0$ and $\op{ind}(C)=2$.
	\item{(iii)} $C$ does not have two positive ends, or two negative ends, at covers of the same simple Reeb orbit.
	\item{(iv)} Let $\gamma$ be a simple Reeb orbit with rotation number $\theta\in\R/\Z$. If $C$ has a positive end at an $m$-fold cover of $\gamma$, then $\op{gcd}(m,\floor{m\theta})=1$. If $C$ has a negative end at an $m$-fold cover of $\gamma$, then $\op{gcd}(m,\ceil{m\theta})=1$.
	\item{(v)} $\M^J_C/\R$ is compact.
\end{description}
Then $\pi_Y(C)\subset Y$ is a global surface of section for the Reeb flow.
\end{proposition}

\subsection{From a holomorphic curve to a foliation}

To prepare for the proof of Proposition~\ref{prop:gss}, we first need to discuss when holomorphic curves in $\R\times Y$ project to embedded surfaces in $Y$, and when the latter foliate subsets of $Y$. Continue to use the notation preceding Proposition~\ref{prop:gss}. If $u$ is a $J$-holomorphic curve in $\R\times Y$, then following Wendl \cite{wendl-jems}, define the {\bf normal Chern number\/} of $u$ by
\[
c_N(u) = \frac{1}{2}\left(2g(u) - 2 + \op{ind}(u) + h_+(u)\right).
\]
The goal of this subsection is to prove the following:
  
\begin{proposition}
\label{prop:foliation}
Let $(Y,\lambda)$ be a nondegenerate contact three-manifold, and let $J$ be a $\lambda$-compatible almost complex structure\footnote{In Proposition~\ref{prop:foliation} it is not necessary to assume that $J$ is generic.} on $\R\times Y$. Let $C$ be a nontrivial irreducible embedded $J$-holomorphic curve in $\R\times Y$ such that:
\begin{description}
\item{(i)} Every $C'\in\M^J_C$ is embedded in $\R\times Y$.
\item{(ii)} $c_N(C)=0$.
\item{(iii)} $C$ does not have two positive ends, or two negative ends, at covers of the same simple Reeb orbit.
\item{(iv)} Let $\gamma$ be a simple Reeb orbit with rotation number $\theta\in\R/\Z$. If $C$ has a positive end at an $m$-fold cover of $\gamma$, then $\op{gcd}(m,\floor{m\theta})=1$. If $C$ has a negative end at an $m$-fold cover of $\gamma$, then $\op{gcd}(m,\ceil{m\theta})=1$.
\end{description}
Then:
\begin{description}
\item{(a)} For every $C'\in\M^J_C$, the projection of $C'$ to $Y$ is an embedding.
\item{(b)} If $g(C)=h_+(C)=0$, then the projections of the curves $C'\in\M^J_C$ to $Y$ give a foliation of an open subset of $Y$.
\end{description}
\end{proposition}

This proposition is a slight generalization of \cite[Prop.\ 3.4]{wh}, and the ideas in the proof go back to \cite{hwz-emb}.

To prove this proposition, we first need to recall the significance of the normal first Chern number. Let $u$ be an immersed $J$-holomorphic curve in $\R\times Y$ with domain $\Sigma$, and let $N\to\Sigma$ denote the normal bundle to $u$. Then $u$ has a well-defined deformation operator
\[
D_u: L^2_1(\Sigma,N) \longrightarrow L^2(\Sigma,T^{0,1}\Sigma\tensor N),
\]
see e.g. \cite[\S2.3]{bn}. The derivative at $u$ of a one-parameter family of curves in $\M^J_u$ defines an element of $\Ker(D_u)$.

If $\psi\in\Ker(D_u)$ does not vanish identically (and $u$ is irreducible), then $\psi$ has only finitely many zeroes, all of which have positive multiplicity; see the review in \cite[Prop.\ 3.4]{wh}. We can then define winding numbers of $\psi$ around the ends of $u$ as follows. Suppose that $u$ has positive ends at $m_i$-fold covers of simple Reeb orbits $\alpha_i$, and negative ends at $n_j$-fold covers of simple Reeb orbits $\beta_j$. Let $\tau$ be a trivialization of $\xi$ over the Reeb orbits $\alpha_i$ and $\beta_j$. Let $\op{wind}_{i,\tau}^+(\psi)$ denote the winding number of $\psi$ around the positive end of $u$ at $\alpha_i^{m_i}$, as measured using the trivialization $\tau$. Likewise, let $\op{wind}_{j,\tau}^-(\psi)$ denote the winding number of $\psi$ around the negative end of $u$ at $\beta_j^{n_j}$ with respect to $\tau$.

It was shown in \cite{hwz-emb}, see the review in \cite[\S5.1]{bn}, that the above winding numbers are bounded by
\begin{equation}
\label{eqn:windCZ}
\begin{split}
\op{wind}_{i,\tau}^+(\psi) &\le \floor{\CZ_\tau\left(\alpha_i^{m_i}\right)/2}\\
\op{wind}_{j,\tau}^-(\psi) &\ge \ceil{\CZ_\tau \left(\beta_j^{n_j}\right)/2}.
\end{split}
\end{equation}
If $\theta_{i,\tau}^+$ denotes the rotation number of $\alpha_i$ with respect to $\tau$, and if $\theta_{j,\tau}^-$ denotes the rotation number of $\beta_j$ with respect to $\tau$, then we can rewrite the above inequalities as
\begin{equation}
\label{eqn:windtheta}
\begin{split}
\op{wind}_{i,\tau}^+(\psi) &\le \floor{m_i\theta_{i,\tau}^+}\\
\op{wind}_{j,\tau}^-(\psi) &\ge \ceil{n_j\theta_{j,\tau}^-}.
\end{split}
\end{equation}

\begin{lemma}
\label{lem:normalcn}
Let $(Y,\lambda)$ be a nondegenerate contact 3-manifold and let $J$ be a $\lambda$-compatible almost complex structure on $\R\times Y$. Let $u$ be an immersed irreducible $J$-holomorphic curve in $\R\times Y$. Suppose that $c_N(u)=0$. Let $\psi$ be a nonzero element of $\Ker(D_u)$. Then:
\begin{description}
\item{(a)}
$\psi$ is nonvanishing.
\item{(b)}
Equality holds in the inequalities \eqref{eqn:windCZ}.
\end{description}
\end{lemma}

\begin{proof}
Let $\tau$ be a trivialization of $\xi$ over the Reeb orbits at whose covers $u$ has ends. The algebraic count of zeroes of $\psi$ is then given by
\begin{equation}
\label{eqn:psi0}
\#\psi^{-1}(0) = c_1(N,\tau) + \op{wind}_{\tau}(\psi)
\end{equation}
where
\begin{equation}
\label{eqn:windtaudef}
\op{wind}_\tau(\psi) = \sum_{i}\op{wind}_{i,\tau}^+(\psi) - \sum_{j}\op{wind}_{j,\tau}^-(\psi).
\end{equation}
As in \cite[Lem.\ 3.1(a)]{pfh2}, we have
\begin{equation}
\label{eqn:oldadj}
c_1(N,\tau) = c_\tau(u) - \chi(\Sigma)
\end{equation}
where $\Sigma$ denotes the domain of $u$.
By \eqref{eqn:windCZ} and \eqref{eqn:windtaudef}, since $\op{CZ}_\tau(\gamma)$ is even if and only if $\gamma$ is positive hyperbolic, we have
\begin{equation}
\label{eqn:2wind}
2\op{wind}_\tau(\psi) \le  \CZ_\tau^{\op{ind}}(u) - p(u) + h_+(u),
\end{equation}
where $p(u)$ denotes the total number of ends of $u$.

Combining \eqref{eqn:psi0}, \eqref{eqn:oldadj}, and \eqref{eqn:2wind}, we obtain
\begin{equation} \label{eqn:zeronchn}
\begin{split}
2\#\psi^{-1}(0) &\le 2c_\tau(u) - 2\chi(\Sigma) + \CZ_\tau^{\op{ind}}(u) - p(u) + h_+(u)\\
&= \op{ind}(u) - \chi(\Sigma) - p(u) + h_+(u)\\
&= 2c_N(u).
\end{split}
\end{equation}

Since $c_N(u)=0$ and all zeroes of $\psi$ have positive multiplicity, we have that $\#\psi^{-1}(0)=0$, establishing (a). To deduce (b), note that the inequality \eqref{eqn:zeronchn} is in fact an equality. This implies that equality also holds in \eqref{eqn:2wind} and hence \eqref{eqn:windCZ}.
\end{proof}

\begin{lemma}
\label{lem:embedding}
Under the hypotheses in Proposition~\ref{prop:foliation}, if $C'\in\M^J_C$, then:
\begin{description}
\item{(a)}
The projection of $C'$ to $Y$ is an embedding.
\item{(b)}
If $C'$ is not obtained from $C$ by translation of the $\R$ factor in $\R\times Y$, then the projections of $C$ and $C'$ to $Y$ are disjoint.
\end{description}
\end{lemma}

\begin{proof}
We proceed in four steps. We continue to use the notation from \eqref{eqn:windCZ} and \eqref{eqn:windtheta}.

{\em Step 1.\/} We first show that the projection of $C'$ to $Y$ is an immersion.

For any $C'\in\M^J_C$, the projection of $\partial_s$ (the derivative of the $\R$ coordinate on $\R\times Y$) to the normal bundle $N$ of $C'$ is a nonzero element of $\Ker(D_{C'})$, since it is the derivative of the family of curves obtained by translating $C'$ in the $\R$ direction. Since we are assuming that $C'$ is not a trivial cylinder, the projection of $\partial_s$ to $N$ does not vanish identically. We have that $c_N(C')=0$ and by hypothesis (i), we have that $C'$ is embedded. Thus, we may apply Lemma~\ref{lem:normalcn}(a) to conclude that the projection of $\partial_s$ to $N$ is nonvanishing. This means that the projection of $C'$ to $Y$ is an immersion. 

{\em Step 2.\/} We next show that if $C'\in\M^J_C$ and $C\neq C'$, then the algebraic count of intersections of $C$ and $C'$ in $\R\times Y$ does not depend on $C'$.

It follows from the definition of $Q_\tau$ in \cite[\S2.7]{ir} that the algebraic count of intersections of $C$ and $C'$ is given by
\[
\#(C\cap C') = Q_\tau(C) + \ell_\tau(C,C'),
\]
where $\ell_\tau(C,C')$ is the ``asymptotic linking number'' of $C$ and $C'$ with respect to $\tau$, defined in \cite[\S2.7]{ir}.

To analyze this asymptotic linking number, let $\op{wind}_{i,\tau}^+(C)$ denote the winding number of $\partial_s$ around the positive end at $\alpha_i^{m_i}$ with respect to $\tau$. Define $\op{wind}_{j,\tau}^-(C)$ likewise for the negative ends. 
As in \cite[Lem.\ 5.5(b)]{bn}, we then have have
\begin{equation}
\label{eqn:linkingbound}
\begin{split}
\ell_\tau(C,C') \le &\sum_{i}m_i\cdot \min(\op{wind}_{i,\tau}^+(C),\op{wind}_{i,\tau}^+(C'))\\
& - \sum_{j}n_j \cdot\max(\op{wind}_{j,\tau}^-(C),\op{wind}_{j,\tau}^-(C')).
\end{split}
\end{equation}
Moreover, equality holds if:
\begin{description}
\item{(*)}
For each $i$, the integers $m_i$ and $\min(\op{wind}_{i,\tau}^+(C),\op{wind}_{i,\tau}^+(C'))$ are relatively prime; and for each $j$, the numbers $n_j$ and $\max(\op{wind}_{j,\tau}^-(C),\op{wind}_{j,\tau}^-(C'))$ are relatively prime.
\end{description}

By Lemma~\ref{lem:normalcn}(b), we have
\begin{equation}
\label{eqn:sharpwind}
\begin{split}
w_{i,\tau}^+(C) &= \floor{m_i\theta_{i,\tau}^+},\\
w_{j,\tau}^-(C) &= \ceil{n_j\theta_{j,\tau}^-}.
\end{split}
\end{equation}
The same holds for any $C'\in\M^J_C$. In particular, by the hypothesis (iv), condition (*) above holds, so equality holds in \eqref{eqn:linkingbound}. Putting all of the above together, we obtain
\[
\#(C\cap C') = Q_\tau(C) + \sum_{i}m_i\floor{m_i\theta_{i,\tau}^+} - \sum_{j} n_j\ceil{n_j\theta_{j,\tau}^-}.
\]
This equation implies that $\#(C\cap C')$ does not depend on the choice of $C'\in\M^J_C$.

{\em Step 3.\/} We now show that if $C'\in\M^J_C$ and $C\neq C'$, then $C$ and $C'$ are disjoint in $\R\times Y$.

As in \cite[Prop.\ 3.4, Step 5]{wh}, hypothesis (iii) and condition (*) above imply that if $C'$ is obtained from $C$ by translating a small amount in the $\R$ direction, then $C$ and $C'$ are disjoint. It then follows from Step 2 that $\#(C\cap C')=0$ for all $C'\in\M^J_C$. By intersection positivity, this means that $C$ and $C'$ are disjoint.

{\em Step 4.\/} We now complete the proof.

(a) We know by Step 1 that the projection of $C$ to $Y$ is an immersion. To show that this projection is an embedding, we just need to show that it is injective. (If this map is an injective immersion, then it is also an embedding because the ends of $C$ are asymptotic to Reeb orbits.) If injectivity fails, then there exist $y\in Y$ and distinct $s_1,s_2\in\R$ such that $(s_1,y),(s_2,y)\in C$. Then $C$ intersects the translation of $C$ by $s_2-s_1$. It follows from Step 3 that $C$ equals the translation of $C$ by $s_2-s_1\neq 0$. This leads to all sorts of contradictions. For example, let $(s^{\ast},y^{\ast}) \in C$ be a point such that $y^{\ast}$ does not lie on a Reeb orbit. For large $n$, we would then have that $(n(s_2-s_1)+s^{\ast},y^{\ast}) \in C$, contradicting asymptotic convergence to Reeb orbits. We conclude that the projection of $C$ to $Y$ is an embedding, and the same argument works for any $C'\in\M^J_C$.

(b) If the projections of $C$ and $C'$ to $Y$ are not disjoint, then there exist $y\in Y$ and $s,s'\in\R$ with $(s,y)\in C$ and $(s',y)\in C'$. Thus $C$ intersects the translation of $C'$ by $s-s'$. It follows from Step 3 that $C$ equals the translation of $C'$ by $s-s'$.
\end{proof}

\begin{proof}[Proof of Proposition~\ref{prop:foliation}.]
 (a) We have proved this in Lemma~\ref{lem:embedding}(a).

(b) If $u$ is any immersed $J$-holomorphic curve in $\R\times Y$ with $c_N(u)<\op{ind}(u)$, then $u$ is automatically cut out transversely; see the review in \cite[Lem.\ 4.1]{dc}, and see \cite{wendl-at} for more general automatic transversality results.  In the present case, $c_N(C)=g(C)=h_+(C)=0$, so $\op{ind}(C)=2$, and the above automatic transversality criterion holds for $C$, as well as for every $C'\in\M^J_C$. Thus $\M^J_C/\R$ is a $1$-manifold.

By Lemma~\ref{lem:embedding}(b), the projections to $Y$ of different elements of $\M^J_C/\R$ are disjoint. To complete the proof that these projections give a foliation of an open subset of $Y$, let $C'\in\M^J_C$, let $(s,y)\in C'$, and let $\pi_Y(C')$ denote the projection of $C'$ to $Y$. We need to show that the natural map
\begin{equation}
\label{eqn:foliationchart}
T_{[C']}(\M^J_C/\R) \longrightarrow (N\pi_Y(C'))_y
\end{equation}
is an isomorphism. Here $N\pi_Y(C')$ denotes the normal bundle to $\pi_Y(C')$ in $Y$. By Lemma~\ref{lem:normalcn}(a), the map
\[
T_{C'}\M^J_C \longrightarrow (NC')_{(s,y)}
\]
is injective, hence an isomorphism. It follows from this that the map \eqref{eqn:foliationchart} is an isomorphism.
\end{proof}

\subsection{From a foliation to a global surface of section}

\begin{proof}[Proof of Proposition~\ref{prop:gss}.]
Note that hypotheses (i), (iii), and (iv) above are the same as the corresponding hypotheses in Proposition~\ref{prop:foliation}, and hypothesis (ii) above implies hypothesis (ii) in Proposition~\ref{prop:foliation}. Then by Lemma~\ref{lem:embedding}(a), the restriction of $\pi_Y$ to $C$, or more generally to any $C'\in\M^J_C$, is an embedding. To complete the proof that $\pi_Y(C)$ is a global surface of section, it is enough to show the following:
\begin{description}
\item{(a)}
For each $C'\in\M^J_C$, the projection $\pi_Y(C')$ is transverse to the Reeb vector field $R$.
\item{(b)}
Let $Z\subset Y$ denote the union of the images of the Reeb orbits at which $C$ has ends. Then for each $y\in Y\setminus Z$, the Reeb trajectory starting at $y$ intersects $\pi_Y(C)$ in both forward and backward time.
\end{description}

We proceed in three steps.

{\em Step 1.\/} We first prove statement (a).

Let $C'\in\M^J_C$. We know from Lemma~\ref{lem:embedding}(a) that $\partial_s$ is nowhere tangent to $C'$. Since $C'$ is $J$-holomorphic and $J\partial_s=R$, it follows that $C'$ is everywhere transverse to the plane spanned by $\partial_s$ and $R$. This implies that $\pi_Y(C')$ is everywhere transverse to $R$.

{\em Step 2.\/}
Let
\[
U = 
\coprod_{C'\in\M^J_C}\pi_Y(C').
\]
We now show that $U=Y\setminus Z$.

We first show that $U\subset Y\setminus Z$. Suppose to get a contradiction that there exist $C'\in\M^J_C$ and $z\in Z$ such that $z\in\pi_Y(C')$. Then by part (a), $\pi_Y(C')$ contains a disk $D$ which intersects $Z$ transversely at $z$. Now $C'$ has an end asymptotic to a Reeb orbit containing $z$, and $\pi_Y$ of points on this end must intersect the disk $D$. Thus $\pi_Y(C')$ is not embedded in $Y$, contradicting Lemma~\ref{lem:embedding}(a).

To prove the reverse inclusion $Y\setminus Z\subset U$, first note that $U$ is an open subset of $Y$, by Proposition~\ref{prop:foliation}(b). Since $Y\setminus Z$ is connected, it is enough to show that any sequence in $U$ has a subsequence that converges to a point in $U$ or a point in $Z$. This holds by our assumption (vi) that $\M^J_C/\R$ is compact.

{\em Step 3.\/} We now prove statement (b).

By assumption (vi) again, we can choose a diffeomorphism 
\[
\phi:\M^J_C/\R\stackrel{\simeq}{\longrightarrow} S^1=\R/\Z.
\]
By Lemma~\ref{lem:embedding}(b) and Step 2,
this induces a function $f:Y\setminus Z\to S^1$ such that $f(y)=\phi([C'])$ when $y\in\pi_Y(C')$. By Proposition~\ref{prop:foliation}(b), the function $f$ is a smooth submersion. By part (a), the derivative $Rf\neq 0$ on all of $Y\setminus Z$. By composing $\phi$ with an orientation-reserving diffeomorphism of $S^1$ if necessary, we may assume that $Rf>0$ on all of $Y\setminus Z$.

Given $y\in Y\setminus Z$ and $T\in\R$, define $g(y,T)\in\R$ to be the total change in $f$ along a Reeb trajectory for time $T$ starting at $y$. It is enough to show that for each $y\in Y\setminus Z$, there exists $T>0$ such that $g(y,T)\ge 1$ and $g(y,-T)\le -1$. In fact, we will show that there is a single $T>0$ which works for all $y\in Y\setminus Z$.

Suppose that $C$ has a positive end at the $m$-fold cover of a simple Reeb orbit $\gamma$. Fix a trivialization $\tau$ of $\xi|_\gamma$, and let $\theta\in\R$ denote the the rotation number of $\gamma$ with respect to $\tau$. Choose an identification of a tubular neighborhood of the image of $\gamma$ in $Y$ with $S^1\times D^2$, such that $\gamma$ is identified with $S^1\times\{0\}$ preserving orientation, and the derivative of the neighborhood identification along $\gamma$ agrees with $\tau$. Let $\rho:\R/m\Z\to \R/\Z=S^1$ denote the projection. By the asymptotics of holomorphic curves reviewed in \cite[\S5.1]{bn}, this end of $C$ is described by a map
\[
\begin{split}
[s_0,\infty)\times (\R/m\Z) & \longrightarrow \R\times (\R/\Z)\times D^2,\\
(s,t) &\longmapsto (s,\rho(t),\eta(s,t)).
\end{split}
\]
Here
\[
\eta(s,t) = e^{-\mu s}\varphi(t) + O\left(e^{(-\mu-\epsilon)s}\right),
\]
where $\varphi:\R/m\Z\to D^2$ is nonvanishing and has winding number less than or equal to $\floor{m\theta}$, and $\mu,\epsilon>0$. More specifically, $\varphi$ is an eigenfunction of the ``asymptotic operator'' associated to $\gamma^m$ with eigenvalue $\mu$; see the review in \cite[Lem.\ 5.2]{bn}.  This means that the Reeb flow near $\gamma$, as it goes $m$ times around $\gamma$, rotates approximately by at least $m\theta-\floor{m\theta}=\{m\theta\}$ relative to the eigenfunction $\varphi$ describing the asymptotics of $C$.

It follows that if $k$ is an integer with $k\{m\theta\}>1$, then for $y$ near the image of $\gamma$, if we set $T=km\mc{A}(\gamma)$, then we have $g(y,T) > 1$ and $g(y,-T) < -1$.  Moreover, since we assumed in (v) that $h_+(C)=0$, we know that $m\theta$ is not an integer, so such a $k$ exists.

Reasoning similarly for the other ends of $C$, we conclude that there exist a neighborhood $V$ of $Z$ and a real number $T_0>0$ such that if $y\in V\setminus Z$, then $g(y,T_0) > 1$ and $g(y,-T_0) < -1$.

By compactness\footnote{One might wish to simplify the proof by finding a positive lower bound on $Rf$ on all of $Y\setminus Z$. However this fails in the generic situation where the first two positive eigenvalues of the asymptotic operator of any Reeb orbit at which $C$ has a positive end are distinct, or the first two negative eigenvalues of the asymptotic operator of any Reeb orbit at which $C$ has a negative end are distinct.}, there exists $\delta>0$ such that the derivative $Rf>\delta$ on $Y\setminus V$. It follows that if we set $T=\delta^{-1} + T_0$, then for every $y\in Y\setminus Z$ we have $g(y,T)>1$ and $g(y,-T) < -1$. To clarify for example why $g(y,T)>1$: if the Reeb flow starting at $y$ stays outside of $V$ for time at least $\delta^{-1}$, then by the definition of $\delta$ we already have $g(y,\delta^{-1})>1$. On the other hand, if for some $\delta'\in[0,\delta^{-1}]$ the image of $y$ under the time $\delta'$ Reeb flow is in $V$, then by the definition of $T_0$ we have $g(y,\delta'+T_0)>1$.
\end{proof}

\subsection{The Poincar\'{e} return map}
\label{sec:prm}

Under the hypotheses of Proposition~\ref{prop:gss}, we can now define the ``Poincar\'e return map''
\begin{equation}
\label{eqn:poincare}
f:\pi_Y(C)\longrightarrow \pi_Y(C)
\end{equation}
as follows: If $y\in Y$, then $f(y)$ is the first intersection with $\pi_Y(C)$ of the forward orbit of $y$ under the Reeb flow. More precisely, for $t\in\R$, let $\phi^t:Y\to Y$ denote the time $t$ Reeb flow. If $y\in\pi_Y(C)$, let $t_+(y)$ denote the infimum over $t>0$ such that $\phi^t(y)\in\pi_Y(C)$. By the ``forward'' part of the third bullet in Definition~\ref{def:gss}, we have $t_+(y)<\infty$. Now define $f(y)=\phi^{t_+(y)}(y)$.

\begin{lemma}
\label{lem:poincare}
Under the hypotheses of Proposition~\ref{prop:gss}:
\begin{description}
\item{(a)}
$d\lambda$ restricts to an area form on $\pi_Y(C)$.
\item{(b)}
The Poincar\'{e} return map \eqref{eqn:poincare} is a diffeomorphism which preserves this area form.
\item{(c)}
The Poincar\'{e} return map preserves the ends of $\pi_Y(C)$.
\end{description}
\end{lemma}

\begin{proof}
By Proposition~\ref{prop:gss}, $\pi_Y(C)$ is a global surface of section for the Reeb flow. By the first bullet in Definition~\ref{def:gss}, the Reeb vector field $R$ is transverse to $\pi_Y(C)$. It follows that (a) holds, and also that $f$ is smooth.

By the ``backward'' part of the third bullet in Definition~\ref{def:gss}, $f$ is a diffeomorphism. And as shown in \cite[Eq.\ (5.10)]{hwz1}, the return map $f$ preserves the area form $d\lambda|_{\pi_Y(C)}$. This proves (b).

To prove (c), observe that the proof of Proposition~\ref{prop:gss} showed that if $C$ has an end at a cover of a simple Reeb orbit $\gamma$, then the Reeb flow, starting a point in $\pi_Y(C)$ near $\gamma$, will return to $\pi_Y(C)$ while staying in a neighborhood of $\gamma$.
\end{proof}

\section{Existence of a special holomorphic curve}
\label{sec:curve}

We would now like to find a holomorphic curve satisfying the criteria in Proposition~\ref{prop:gss}, so that it projects to a global surface of section for the Reeb flow. The goal of this section is to prove Proposition~\ref{prop:special} below, which asserts that we can do this, under the assumptions that $c_1(\xi)$ is torsion and that there are only finitely many simple Reeb orbits. In fact, we will obtain a curve satisfying even more properties than those required for Proposition~\ref{prop:gss}, namely: 

\begin{definition}
\label{def:special}
Let $(Y,\lambda)$ be a nondegenerate contact three-manifold, and let $J$ be a $\lambda$-compatible almost complex structure on $\R\times Y$.
A $J$-holomorphic curve $C$ in $\R\times Y$ is {\bf special\/} if it has the following properties:
\begin{description}
\item{(a)}
$\op{ind}(C)=I(C)=2$, and $C$ is irreducible and embedded in $\R\times Y$.
\item{(b)}
$C$ has at least one positive end, and at least one negative end, at elliptic Reeb orbits.
\item{(c)}
$C$ has genus zero and at most 3 ends.
\item{(d)}
$C$ does not have two positive ends, or two negative ends, at covers of the same simple Reeb orbit.
\item{(e)}
$C$ does not have any ends at hyperbolic orbits, except possibly one end at a simple negative hyperbolic orbit.
\item{(f)}
The component of the moduli space of $J$-holomorphic curves containing $C$ is compact.
\end{description}
\end{definition}

\begin{proposition}
\label{prop:special}
Let $(Y,\lambda)$ be a nondegenerate contact three-manifold with $c_1(\xi)\in H^2(Y;\Z)$ torsion and with only finitely many simple Reeb orbits. Let $J$ be a generic $\lambda$-compatible almost structure on $\R\times Y$. Then there exists a special $J$-holomorphic curve in $\R\times Y$.
\end{proposition}

\subsection{A sequence of U-curves}

The first step in the proof of Proposition~\ref{prop:special} is to obtain a sequence of $U$-curves with some control over their $J_0$ index. 

To prepare for this, note from \eqref{eqn:defI} and \eqref{eqn:defJ0} that if $\alpha$ and $\beta$ are any orbit sets with $[\alpha]=[\beta]$ and if $Z\in H_2(Y,\alpha,\beta)$, then the difference between $I$ and $J_0$ is given by
\begin{equation}
\label{eqn:IMinusJ}
I(\alpha,\beta,Z) - J_0(\alpha,\beta,Z) = 2 c_{\tau}(Z) + \sum_i \op{CZ}_{\tau}(\alpha_i^{m_i}) - \sum_j \op{CZ}_{\tau}(\beta_j^{n_j}).
\end{equation}
We will also need the following linearity property of the relative first Chern class: Let $\alpha'$ and $\beta'$ be another pair of orbit sets with $[\alpha']=[\beta']$, and let $Z'\in H_2(Y,\alpha',\beta')$. Then $Z+Z'\in H_2(Y,\alpha\alpha',\beta\beta')$ is defined; here $\alpha\alpha'$ denotes the ``product'' orbit set obtained by taking the union of the simple Reeb orbits in $\alpha$ and $\alpha'$ and adding their multiplicities.  Let $\tau$ be a trivialization of $\xi$ over all the Reeb orbits in the four orbit sets $\alpha,\beta,\alpha',\beta'$; it then follows from the definition of $c_\tau$ in \cite[\S3.2]{bn} that
\begin{equation}
\label{eqn:linearity}
c_\tau(Z+Z') = c_\tau(Z) + c_\tau(Z').
\end{equation}

\begin{lemma}
\label{lem:IJ0}
Let $(Y,\lambda)$ be a nondegenerate contact 3-manifold with $c_1(\xi)\in H^2(Y;\Z)$ torsion. Then:
\begin{description}
\item{(a)} There is a unique way to assign, to each orbit set $\alpha$ with $[\alpha]=0\in H_1(Y)$, integers $I(\alpha)$ and $J_0(\alpha)$, such that (i) $I(\emptyset)=J_0(\emptyset)=0$, and (ii) if $\beta$ is another orbit set with $[\beta]=0$, then for any $Z\in H_2(Y,\alpha,\beta)$, we have
\begin{equation}
\label{eqn:IJ0abs}
\begin{split}
I(\alpha,\beta,Z) & = I(\alpha) - I(\beta),\\
J_0(\alpha,\beta,Z) & = J_0(\alpha) - J_0(\beta).
\end{split}
\end{equation}
\item{(b)} If there are only finitely many simple Reeb orbits, then there is a constant $\delta_1 > 0$ such that if $\alpha$ is any orbit set with $[\alpha]=0$, then
\begin{equation}
\label{eqn:IJ0diff}
|I(\alpha) - J_0(\alpha)| \le \delta_1 \mc{A}(\alpha).
\end{equation}
\end{description}
\end{lemma}

\begin{proof}
(a)
We must define $I(\alpha) = I(\alpha,\emptyset,Z)$ and $J_0(\alpha) = J_0(\alpha,\emptyset,Z)$ where $Z$ is any class in $H_2(Y,\alpha,\emptyset)$. These definitions do not depend on the choice of $Z$ in view of \eqref{eqn:indexambiguity} and \eqref{eqn:J0ambiguity} since $[\alpha]=0$ and $c_1(\xi)$ is torsion.
The equations \eqref{eqn:IJ0abs} hold as a result of the additivity property \eqref{eqn:Iadditive} of $I$ and an analogous property of $J_0$.

(b) Let $\alpha_1,\ldots,\alpha_n$ denote the simple Reeb orbits. Fix a trivialization $\tau$ of $\xi$ over $\alpha_1,\ldots,\alpha_n$. Let $\alpha=\prod_i\alpha_i^{m_i}$ be a nullhomologous orbit set. Define $c_\tau(\alpha)=c_\tau(\alpha,\emptyset,Z)$ for any $Z\in H_2(Y,\alpha,\emptyset)$; this is well defined by \eqref{eqn:ctauambiguity}. Then by part (a) and \eqref{eqn:IMinusJ}, we have
\begin{equation}
\label{eqn:imj}
I(\alpha) - J_0(\alpha) = 2c_\tau(\alpha) + \sum_{i=1}^n\op{CZ}_\tau(\alpha_i^{m_i}).
\end{equation}
Here we interpret $\op{CZ}_\tau(\alpha_i^{m_i})=0$ when $m_i=0$. 

To analyze the $c_\tau$ term in \eqref{eqn:imj}, note that $\alpha=\prod_{i=1}^n\alpha_i^{m_i}$ is nullhomologous if and only if $(m_1,\ldots,m_n)$ is an element of the set
\[
W = \left\{(m_1,\ldots,m_n)\in \N^n \;\bigg|\; \sum_{i=1}^nm_i[\alpha_i]=0\in H_1(Y)\right\}.
\]
Thus $c_\tau$ defines a map $W\to \Z$. Let $W_\Q$ denote the span of $W$ in $\Q^n$. The map $c_\tau:W\to\Z$ is additive by \eqref{eqn:linearity}, hence $c_\tau$ extends uniquely to a linear map $W_\Q\to\Q$.
This extension is then given by the inner product with a fixed vector in $W_\Q$. We conclude that there are constants $w_1,\ldots,w_n\in\Q$ such that every nullhomologous orbit set $\alpha=\prod_{i=1}^n\alpha_i^{m_i}$ satisfies
\begin{equation}
\label{eqn:analyze1}
c_\tau(\alpha) = \sum_{i=1}^nw_im_i.
\end{equation}

To estimate the Conley-Zehnder index term in \eqref{eqn:imj}, note that by \eqref{eqn:CZequation} we have
\[
\op{CZ}_\tau\left(\alpha_i^{m_i}\right) = \floor{m_i\theta_i} + \ceil{m_i\theta_i}
\]
where $\theta_i$ denotes the rotation number of $\alpha_i$ with respect to $\tau$. In particular,
\begin{equation}
\label{eqn:analyze2}
\left|\op{CZ}_\tau\left(\alpha_i^{m_i}\right)\right| \le 2\ceil{|\theta_i|}m_i.
\end{equation}

Combining \eqref{eqn:imj}, \eqref{eqn:analyze1}, and \eqref{eqn:analyze2}, we obtain
\begin{equation}
\label{eqn:imja}
|I(\alpha) - J_0(\alpha)| \le \sum_{i=1}^n d_i m_i
\end{equation}
where $d_i=2(|w_i|+\ceil{|\theta_i|})$.
In addition, we have
\begin{equation}
\label{eqn:actionobvious}
\mc{A}(\alpha) = \sum_{i=1}^na_im_i
\end{equation}
where $a_i=\mc{A}(\alpha_i)>0$.  It follows from \eqref{eqn:imja} and \eqref{eqn:actionobvious} that the estimate \eqref{eqn:IJ0diff} holds with $\delta_1=\max(d_i/a_i)$.
\end{proof}

\begin{lemma}
\label{lem:Usequence}
Let $(Y,\lambda)$ be a nondegenerate contact 3-manifold with only finitely many simple Reeb orbits and with $c_1(\xi)\in H^2(Y;\Z)$ torsion. Let $J$ be a generic $\lambda$-compatible almost complex structure. Let $\epsilon>0$. Then at least one of the following is true:
\begin{description}
\item{(1)}
There exist ECH generators $\alpha$ and $\beta$, and a $U$-curve $\mc{C}\in\M^J(\alpha,\beta)$, such that $[\alpha]=[\beta]=0\in H_1(Y)$ and $\mc{A}(\alpha)-\mc{A}(\beta) < \epsilon$ and $J_0(\mc{C}) \le 1$.
\item{(2)}
For every positive integer $l$, there exist ECH generators $\alpha(0),\alpha(1),\ldots,\alpha(l)$ with $[\alpha(i)]=0\in H_1(Y)$, and $U$-curves $\mc{C}(i)\in\M^J(\alpha(i),\alpha(i-1))$ for $i=1,\ldots,l$, such that for each $i$ we have
\begin{gather}
\label{eqn:lowaction}
\mc{A}(\alpha(i)) - \mc{A}(\alpha(i-1)) < \epsilon,\\
\label{eqn:lowJ0}
J_0(\mc{C}(i)) = 2.
\end{gather}
\end{description}
\end{lemma}

\begin{proof}
Suppose that (1) is false, in particular that every $U$-curve $\mc{C}\in\M^J(\alpha,\beta)$ with $[\alpha]=[\beta]=0$ and $\mc{A}(\alpha) - \mc{A}(\beta) < \epsilon$ satisfies $J_0(\mc{C})\ge 2$. We must prove that (2) is true.

By Proposition~\ref{prop:Useq}(a) and our assumption that $c_1(\xi)$ is torsion, there exists a $U$-sequence $\{\sigma_k\}_{k\ge 1}$ in $ECH(Y,\lambda,\Gamma)$ with $\Gamma=0$. By \eqref{eqn:Useqasymptotics}, there is a constant $\delta_2 > 0$ such that
\[
c_{\sigma_k}(Y,\lambda) \le \delta_2 k^{1/2}.
\]

Recall from \S\ref{sec:Ucurves} that the map $U$ on $ECH(Y,\lambda,0)$ is induced by a chain map $U_{J,z}$ on $ECC(Y,\lambda,0)$ counting $U$-curves passing through a point $z\in Y$ which is not on any Reeb orbit. Fix a large positive integer $k$. By the definition of $c_{\sigma_k}$, the class $\sigma_k$ can be represented by a cycle $x=\sum_jx_j$ in the chain complex $ECC(Y,\lambda,0)$ such that each $x_j$ is an ECH generator with action $\mc{A}(x_j) \le c_{\sigma_k}$. Since $U^{k-1}\sigma_k\neq 0$ on homology, it follows that $U_{J,z}^{k-1}x\neq 0$ on the chain level. Thus there are ECH generators $\alpha(1),\ldots,\alpha(k)$ such that $[\alpha(i)]=0\in H_1(Y)$ for each $i$, the ECH generator $\alpha(k)$ is one of the $x_j$, and $\langle U_{J,z}\alpha(i),\alpha(i-1)\rangle\neq 0$ for $i=2,\ldots,k$. In particular,
\begin{equation}
\label{eqn:ipa}
\mc{A}(\alpha(k))\le \delta_2k^{1/2},
\end{equation}
and there exist $U$-curves $\mc{C}(i)\in\M^J(\alpha(i),\alpha(i-1))$ for $i=2,\ldots,k$.

We claim that for every positive integer $l$, if $k$ is sufficiently large, then there will be $l$ consecutive integers $i$ from $2$ to $k$ satisfying \eqref{eqn:lowaction} and \eqref{eqn:lowJ0}. It is enough to show that there are at most $O(k^{1/2})$ integers $i$ from $2$ to $k$ such that \eqref{eqn:lowaction} and \eqref{eqn:lowJ0} are not both satisfied.

By \eqref{eqn:ipa}, there are at most $\epsilon^{-1}\delta_2 k^{1/2}$ integers $i$ from $2$ to $k$ such that \eqref{eqn:lowaction} is not satisfied. By our hypothesis, this also implies that there are most $\epsilon^{-1}\delta_2 k^{1/2}$ integers $i$ from $2$ to $k$ such that $J_0(\mc{C}(i)) \le 1$.

Since the $\mc{C}(i)$ are $U$-curves, they each have ECH index $2$, so
\[
\sum_{i=2}^kI(\mc{C}(i)) = 2(k-1).
\]
It then follows from Lemma~\ref{lem:IJ0}(b) and the estimate \eqref{eqn:ipa} that
\[
\begin{split}
\sum_{i=2}^kJ_0(\mc{C}(i)) &= J_0(\alpha(k)) - J_0(\alpha(1))\\
& \le 2(k-1) + 2\delta_1\delta_2k^{1/2}.
\end{split}
\]
Recall that $J_0(\mc{C}(i))$ is an integer. Also, it follows from Proposition~\ref{prop:obscure} that $J_0(\mc{C}(i))\ge -1$ for each $i$. We deduce that there are at most $(2\delta_1 + 4\epsilon^{-1})\delta_2k^{1/2}$ integers $i$ from $2$ to $k$ such that \eqref{eqn:lowJ0} is not satisfied.
\end{proof}

\begin{remark}
\label{rem:technical}
The proof of Lemma~\ref{lem:Usequence} is the part of the proof of Theorem~\ref{thm:main} where we make essential use of the assumption that $c_1(\xi)$ is torsion. Without this assumption, we can still find $U$-curves in $\mc{M}^J(\alpha,\beta)$ with $\mc{A}(\alpha)-\mc{A}(\beta)<\epsilon$, such that $[\alpha]=[\beta]=\Gamma$ where $c_1(\xi) + 2\op{PD}(\Gamma)\in H^2(Y;\Z)$ is torsion. However we do not know how to control $J_0$ of these curves, because when $c_1(\xi)$ is not torsion, $J_0$ of these curves no longer depends only on $\alpha$ and $\beta$, but also on their relative homology classes by \eqref{eqn:J0ambiguity}. How to bound $J_0$ in this case is an interesting question for future research.
\end{remark}

\subsection{Controlling topological complexity}

We now use Proposition~\ref{prop:obscure} to describe the possible structure of a U-curve $\mc{C}=\mc{C}_0\sqcup C_1$ with $J_0(\mc{C}) \le 2$.

\begin{lemma}
\label{lem:J01}
Under the hypotheses of Proposition~\ref{prop:obscure}, suppose that $J_0(\mc{C})\le 1$ and that $C_1$ has at least one negative end. Then:
\begin{description}
\item{(a)} $C_1$ has genus zero and at most 3 ends.
\item{(b)} $C_1$ does not have two positive ends or two negative ends at covers of the same simple Reeb orbit.
\end{description}
\end{lemma}

\begin{proof}
It follows from equation \eqref{eqn:obscure} that $\chi(C_1)\ge -1$. Since $C_1$ always has at least one positive end, and we are assuming that $C_1$ has at least one negative end, assertion (a) follows.

If assertion (b) is false, then it follows from equation \eqref{eqn:obscure} that $\chi(C_1)\ge 0$, which is impossible since now $C_1$ has at least 3 ends.
\end{proof}

\begin{lemma}
\label{lem:J02}
Under the hypotheses of Proposition~\ref{prop:obscure}, suppose that $J_0(\mc{C})= 2$ and that $C_1$ has at least one negative end. Then:
\begin{description}
\item{(a)} If for some $i$, both $\mc{C}_0$ and $C_1$ have positive ends at covers of $\alpha_i$, or if for some $j$, both $\mc{C}_0$ and $C_1$ have negative ends at covers of $\beta_j$, then $C_1$ satisfies the conclusions of Lemma~\ref{lem:J01}.
\item{(b)} $C_1$ has at most $2$ positive ends at covers of $\alpha_i$ for each $i$, and at most $2$ negative ends at covers of $\beta_j$ for each $j$.
\item{(c)} 
If $C_1$ has $2$ positive ends at covers of $\alpha_i$ for some $i$, then $C_1$ has exactly one negative end. Likewise, if $C_1$ has $2$ negative ends at covers of $\beta_j$ for some $j$, then $C_1$ has exactly one positive end.
\end{description}
\end{lemma}

\begin{proof}
(a) This follows from equation \eqref{eqn:obscure} as in the proof of Lemma~\ref{lem:J01}.

(b) If $C_1$ has more than two positive ends at covers of $\alpha_i$ for some $i$, or if $C_1$ has more than two negative ends at covers of $\beta_j$ for some $j$, then equation \eqref{eqn:obscure} implies that $\chi(C_1)\ge 0$, which is a contradiction since in this case $C_1$ has at least $4$ ends.

(c) If $C_1$ has $2$ positive ends at covers of $\alpha_i$ for some $i$, then by equation \eqref{eqn:obscure}, $\chi(C_1)\ge -1$. Since we are assuming that $C_1$ has at least one negative end, it follows that $C_1$ has exactly one negative end.  The proof in the case where $C_1$ has $2$ negative ends at covers of $\beta_j$ for some $j$ is analogous.
\end{proof}

\subsection{Exceptional Reeb orbits}

\begin{definition}
\label{def:exceptional}
Let $(Y,\lambda)$ be a nondegenerate contact three-manifold. Let $\gamma$ be a simple Reeb orbit and let $m$ be a positive integer. We say that the pair $(\gamma,m)$ is {\bf exceptional\/}  if $|p_\gamma^+(m)| + |p_\gamma^-(m)| \le 3$.
\end{definition}

\begin{lemma}
\label{lem:exceptional}
Let $(Y,\lambda)$ be a nondegenerate contact three-manifold, and let $\gamma$ be a simple Reeb orbit. Then there are only finitely many positive integers $m$ such that the pair $(\gamma,m)$ is exceptional.
\end{lemma}

\begin{proof}

If $\gamma$ is hyperbolic, then Lemma~\ref{lem:exceptional} follows directly from \eqref{eqn:partitionh+} and \eqref{eqn:partitionh-}.  So assume that $\gamma$ is elliptic.  We need to show that if $\theta$ is an irrational number then there are only finitely many positive integers $m$ with $|p_\theta^+(m)| + |p_\theta^-(m)|\in\{2,3\}$. Since $p_\theta^-(m) = p_{-\theta}^+(m)$, it is enough to show that there are only finitely many positive integers $m$ with $p_\theta^+(m)=(m)$ and $|p_\theta^-(m)|\in\{1,2\}$.  

\medskip

{\em Claim\footnote{More generally, if $\theta$ is irrational and $m>1$ then $1\notin p_\theta^+(m) \Leftrightarrow 1\in p_\theta^-(m)$. This is shown in \cite[Eq.\ (22)]{pfh2}, and can also be proved similarly to the proof of the Claim.}.\/} If $\theta$ is irrational and $p_\theta^+(m)=(m)$, then $1\in p_\theta^-(m)$.

\medskip

{\em Proof of Claim.\/} Let $a$ denote the smallest element of $p_\theta^-(m)$. By the definition of $p_\theta^-$ in \S\ref{sec:partition}, the triangle with vertices $(m,\floor{m\theta})$, $(m,\ceil{m\theta})$, and $(m-a,\ceil{(m-a)\theta})$ does not contain any lattice points other than its vertices. Thus by Pick's theorem, this triangle has area $1/2$. But this triangle also has area $a/2$, so $a=1$.

\medskip

It follows from the Claim that
if $p_\theta^+(m)=p_\theta^-(m)=(m)$ then $m=1$. 

It remains to show that there are only finitely many $m$ with $p_\theta^+(m)=(m)$ and $|p_\theta^-(m)|=2$. In this case, it follows from the Claim that $p_\theta^-(m)=(m-1,1)$. Then by the definition of $p_\theta^+$ and $p_\theta^-$, the quadrilateral with vertices $(0,0)$, $(m,\floor{m\theta})$, $(m,\ceil{m\theta})$, and $(m-1,\ceil{(m-1)\theta})$ contains no lattice points other than its vertices. This quadrilateral contains the triangle with vertices $(0,0)$, $(m,\floor{m\theta})$, and $(m-1,\ceil{(m-1)\theta})$, so that triangle also contains no lattice points other than its vertices, and thus has area $1/2$. Recomputing this area using determinants, we obtain
\[
m\ceil{(m-1)\theta} - (m-1)\floor{m\theta}=1.
\]
Since $p_\theta^-(m)=(m-1,1)$, we also know from \cite[Ex.\ 3.13(b)]{bn} that
\[
\ceil{(m-1)\theta} + \ceil{\theta} = \ceil{m\theta}.
\]
Substituting this equation into the previous one, we obtain
\[
\begin{split}
1 &= m(\ceil{m\theta}-\ceil{\theta}) - (m-1)\floor{m\theta}\\
&= (m-1) + \ceil{m\theta} - m\ceil{\theta}\\
&= (m-1) + m\theta + 1 - \{m\theta\} - m(\theta + 1 - \{\theta\})\\
&= m\{\theta\} - \{m\theta\}.
\end{split}
\]
This implies that $m\{\theta\}<2$, so there are only finitely many such $m$.
\end{proof}

\subsection{Low energy curves}

\begin{lemma}
\label{lem:lowenergy}
Let $(Y,\lambda)$ be a nondegenerate contact three-manifold with only finitely many simple Reeb orbits. Then there exists a constant $\epsilon>0$ with the following property. Let $\alpha$ and $\beta$ be ECH generators with $\mc{A}(\alpha) - \mc{A}(\beta) < \epsilon$. Let $J$ be a generic $\lambda$-compatible almost complex structure on $\R\times Y$. Let $\mc{C}=\mc{C}_0\sqcup C_1\in\M^J(\alpha,\beta)$ be a $U$-curve. Then:
\begin{description}
\item{(a)} Let $\alpha'$ and $\beta'$ denote the orbit sets for which $C_1\in\M^J(\alpha',\beta')$. Then there is at least one pair $(\gamma,m)\in\alpha'$ such that $(\gamma,m)$ is not exceptional, and likewise there is at least one nonexceptional pair $(\gamma,m)\in \beta'$.
\item{(b)}
$C_1$ has at least one positive end and at least one negative end at elliptic Reeb orbits.
\item{(c)}
If $C_1$ has genus 0 and at most 3 ends, then the component of the moduli space of $J$-holomorphic curves containing $C_1$ is compact.
\end{description}
\end{lemma}

\begin{proof}
Since there are only finitely many simple Reeb orbits, and since by Lemma~\ref{lem:exceptional} there are only finitely many exceptional pairs $(\gamma,m)$  there are only finitely many orbit sets $x$ where every pair $(\gamma,m) \in x$ is exceptional.  It follows that we can choose $\epsilon>0$ such that the following holds:
\begin{description}
\item{(i)}
If $x$ is an orbit set such that every pair $(\gamma,m)\in x$ is exceptional, and if $y$ is another orbit set with $\mc{A}(x)\neq \mc{A}(y)$, then $|\mc{A}(x) - \mc{A}(y)| > \epsilon$.
\end{description}
In particular, we also have:
\begin{description}
\item{(ii)}
If $\gamma$ is a simple Reeb orbit then $\mc{A}(\gamma) > \epsilon$.
\item{(iii)}
If $\gamma$ is a positive hyperbolic Reeb orbit which is either simple or a double cover of a simple negative hyperbolic orbit, and if $\gamma'$ is any other Reeb orbit with $\mc{A}(\gamma)\neq\mc{A}(\gamma')$, then $|\mc{A}(\gamma) - \mc{A}(\gamma')| > \epsilon$. 
\end{description}
We claim that properties (i)-(iii) above imply assertions (a)--(c). To see this, let $J$ be a generic $\lambda$-compatible almost complex structure, and let $\mc{C} = \mc{C}_0\sqcup C_1\in\M^J(\alpha,\beta)$ be a $U$-curve with $\mc{A}(\alpha)-\mc{A}(\beta)<\epsilon$.
We can write $\alpha=\alpha'\gamma$ and $\beta=\beta'\gamma$, where $\gamma$ is the orbit set such that $\mc{C}_0=\R\times\gamma$, and $C_1\in\M^J(\alpha',\beta')$. Then we also have $0 < \mc{A}(\alpha') - \mc{A}(\beta') < \epsilon$. We now prove (a)--(c) as follows.

(a) 
This follows immediately from property (i).

(b)
Suppose to get a contradiction that (b) does not hold. Without loss of generality, $C_1$ does not have a positive end at an elliptic Reeb orbit. This means that all orbits in $\alpha'$ are hyperbolic. Then, since $\alpha$ is an ECH generator, all orbits in $\alpha'$ have multiplicity one. In particular, every element of $\alpha'$ is exceptional. Since $C_1$ is nontrivial, $\mc{A}(\alpha') \neq \mc{A}(\beta')$. This contradicts property (i) with $x=\alpha'$ and $y=\beta'$.

(c) Suppose to get a contradiction that $C_1$ has genus 0 and at most 3 ends, but the component of the moduli space of $J$-holomorphic curves containing $C_1$ is not compact.  Then by the compactness theorem in \cite[Lem. 5.11]{bn}, there exists a sequence of $J$-holomorphic curves in the moduli space containing $C_1$ which converges in an appropriate sense to a ``broken $J$-holomorphic current'' from $\alpha'$ to $\beta'$ with more than 1 level and with total ECH index 2. This broken $J$-holomorphic current is a $k$-tuple of $J$-holomorphic currents $(\mc{C}(1),\ldots,\mc{C}(k))$ where $k>1$, the current $\mc{C}(i)\in\M^J(\alpha(i),\alpha(i-1))$ is not a union of covers of trivial cylinders, $\alpha(k)=\alpha'$ and $\alpha(0)=\beta'$, and $\sum_{i=1}^kI(\mc{C}(i))=2$. By Proposition~\ref{prop:lowI}, it follows that $k=2$ and $I(\mc{C}(1))=I(\mc{C}(2))=1$.

Write $\mc{C}^+=\mc{C}(2)$ and $\mc{C}^-=\mc{C}(1)$. By Proposition~\ref{prop:lowI}(1), we can write
\[
\begin{split}
\mc{C}^+ &=\mc{C}^+_0\sqcup C^+_1\in\M^J(\alpha',\eta),\\
\mc{C}^- & = \mc{C}^-_0\sqcup C^-_1 \in\M^J(\eta,\beta')
\end{split}
\]
for some (not necessarily admissible) orbit set $\eta$, where each component of $\mc{C}^{\pm}_0$ is a trivial cylinder, while $C^{\pm}_1$ is embedded and has $\op{ind}(C^{\pm}_1) = I(C^{\pm}_1) = 1$.  Note that $\mc{A}(\alpha') - \mc{A}(\eta)$ and $\mc{A}(\eta) - \mc{A}(\beta')$ are both less than $\epsilon$.

By property (ii), $C^+_1$ and $C^-_1$ each have at least one negative end (and of course at least one positive end). In particular, $\chi(C^+_1), \chi(C^-_1) \le 0$.

Now at least one of $C^+_1, C^-_1$ must be a cylinder. Otherwise $\chi(C^+_1), \chi(C^-_1)\le -1$, so $\chi(C_1)\le -2$ (by the definition of convergence to a broken holomorphic current in \cite[\S5.3]{bn}), contradicting our assumption that $C_1$ has genus 0 and at most 3 ends.

Since the cylinder $C^+_1$ or $C^-_1$ has Fredholm index $1$, the two Conley-Zehnder terms in \eqref{eqn:Fredholmindex} must have opposite parity, which means that one end is at a positive hyperbolic orbit, while the other end is at an elliptic or negative hyperbolic orbit. Since this cylinder also has ECH index 1, it follows from the partition conditions in Proposition~\ref{prop:partitions} and \eqref{eqn:partitionh+} and \eqref{eqn:partitionh-} that the positive hyperbolic orbit is either simple, or the double cover of a negative hyperbolic orbit. The existence of this cylinder now contradicts condition (iii).
\end{proof}

\subsection{Existence of a special curve}
\label{sec:esc}

\begin{proof}[Proof of Proposition~\ref{prop:special}.] Suppose there are $n$ simple Reeb orbits. We now invoke Lemma~\ref{lem:Usequence}, with the constant $\epsilon$ provided by Lemma~\ref{lem:lowenergy}.

Suppose that case (1) in Lemma~\ref{lem:Usequence} holds. We then have a $U$-curve $\mc{C}=\mc{C}_0\sqcup C_1\in\M^J(\alpha,\beta)$ with $\mc{A}(\alpha) - \mc{A}(\beta)<\epsilon$ and $J_0(\mc{C})\le 1$. We claim that the curve $C_1$ is special. Condition (a) in the definition of ``special'' holds because $\mc{C}$ is a $U$-curve. Condition (b) holds by Lemma~\ref{lem:lowenergy}(b) above. Conditions (c) and (d) then hold by Lemma~\ref{lem:J01}. Since $C_1$ has at most three ends, at least two of which are at elliptic orbits, and since $C_1$ has even Fredholm index, it follows that $C_1$ cannot have an end at a positive hyperbolic Reeb orbit. Since $\alpha$ and $\beta$ are ECH generators, if $C_1$ has an end at a negative hyperbolic Reeb orbit, then this orbit is simple. This proves condition (e) in the definition of ``special''. Condition (f) holds by Lemma~\ref{lem:lowenergy}(c).

Suppose now that case (2) in Lemma~\ref{lem:Usequence} holds. We can then put $l=2n+1$ into Lemma~\ref{lem:Usequence} to obtain ECH generators $\alpha(0),\ldots,\alpha(2n+1)$, and $U$-curves
\[
\mc{C}(i) = \mc{C}(i)_0 \sqcup C(i)_1 \in \M^J(\alpha(i),\alpha(i-1))
\]
for $i=1,\ldots,2n+1$, such that \eqref{eqn:lowaction} and \eqref{eqn:lowJ0} hold for each $i$. We claim that for at least one $i$, the curve $C(i)_1$ is special. We know that for each $i$, the curve $C(i)_1$ satisfies condition (a) in the definition of ``special'' since $\mc{C}(i)$ is a $U$-curve, and also condition (b) by Lemma~\ref{lem:lowenergy}(b). We need to show that for at least one $i$, the curve $C(i)_1$ also satisfies conditions (c) and (d); the conditions (e) and (f) will then follow as before.

We claim that for some $i$, the curves $\mc{C}(i)_0$ and $C(i)_1$ have positive ends at covers of the same simple orbit, or negative ends at covers of the same simple orbit. Then by Lemma~\ref{lem:J02}(a), the curve $C(i)_1$ satisfies conditions (c) and (d) in the definition of ``special'', and we are done.

To prove the claim, suppose that for all $i$, the curves $\mc{C}(i)_0$ and $C(i)_1$ do not have positive ends at covers of the same simple orbit, or negative ends at covers of the same simple orbit. We then obtain a contradiction as follows.

Let us call the curve $C(i)_1$ ``Type I'' if it does not have two negative ends at covers of the same simple orbit. Call the curve $C(i)_1$ ``Type II'' if it does not have two positive ends at covers of the same simple orbit. By Lemma~\ref{lem:J02}(c), each curve $C(i)_1$ is Type I or Type II (or possibly both).

Suppose that for some $i>1$, the curve $C(i)_1$ is Type I. Then by Lemma~\ref{lem:lowenergy}(a), there is a simple orbit $\gamma$ such that $C(i)_1$ has exactly one negative end at a cover of $\gamma$, of multiplicity $m$, and the pair $(\gamma,m)$ is not exceptional. Then the curve $C(i-1)_1$ cannot have any positive ends at covers of $\gamma$, by Lemma~\ref{lem:J02}(b) and the definition of ``exceptional'', hence the component of $C(i-1)_0$ along $\gamma$ must have multiplicity $m$.  It then follows by downward induction on $j$ that if $1\le j<i$, then the curve $C(j)_1$ cannot have any positive or negative ends at covers of $\gamma$.

Likewise, if for some $i<2n+1$, the curve $C(i)_1$ is Type II, then there is a simple orbit $\gamma$ such that $C(i)_1$ has a positive end at a cover of $\gamma$, but for $i<j\le 2n+1$, the curve $C(j)_1$ cannot have any positive or negative ends at covers of $\gamma$.

Now among the $2n-1$ curves $C(2)_1,\ldots,C(2n)_{1}$, at least $n$ of them have Type I, or at least $n$ of them have Type II (or possibly both). In the first case, there are no possible orbits at which $C(1)_1$ can have ends, which is the desired contradiction. In the second case, there are no possible orbits at which $C(2n+1)_1$ can have ends, which is likewise a contradiction.
\end{proof}

\section{Two or infinitely many Reeb orbits}
\label{sec:franks}

In this section we complete the proof of the main theorem~\ref{thm:main}.

\subsection{Invoking a theorem of Franks}

To prove Theorem~\ref{thm:main}, we will need one more dynamical fact.  

\begin{proposition}
\label{prop:usefranks}
Let $C$ be a surface diffeomorphic to $S^2$ with $k$ points removed.  Let $\omega$ be an area form on $C$ with $\int_C\omega<\infty$.  Let $f:(C,\omega)\to (C,\omega)$ be an area-preserving diffeomorphism which acts as the identity on the set of ends.
Then:
\begin{itemize}
\item If $k=2$, then $f$ has either no periodic orbits or infinitely many.
\item If $k>2$, then $f$ has infinitely many periodic orbits.
\end{itemize}
\end{proposition}

To prove Proposition~\ref{prop:usefranks}, we will use the following theorem of Franks. Below, let $A$ denote the open annulus $\mathring{D^2}\setminus\{0\}$.

\begin{theorem}
\label{thm:franks}
\cite[Thm. 4.4]{franks}
Let $f:A\to A$ be a homeomorphism which preserves Lebesgue measure. If $f$ has at least one periodic orbit, then $f$ has infinitely many periodic orbits.
\end{theorem}

To apply Theorem~\ref{thm:franks}, we will need the following result of Berlanga-Epstein, generalizing results of Oxtoby-Ulam \cite{ou}. Let $\mu$ be a Borel measure on a manifold $X$. We say that $\mu$ is ``non-atomic'' if $\mu(\{x\})=0$ for each $x\in X$, and that $\mu$ has ``full support'' if $\mu(U)>0$ for every nonempty open set $U\subset X$.

\begin{theorem}
\label{thm:berep}
(special case of \cite{berep})
Let $\mu_1$ and $\mu_2$ be two Borel measures on a manifold $X$ which are non-atomic and have full support. If $\mu_1(X)=\mu_2(X)<\infty$, then there is a homeomorphism $h:X\to X$ such that $h_*\mu_1=\mu_2$.
\end{theorem}

\begin{proof}[Proof of Proposition~\ref{prop:usefranks}.]
We can identify $C$ with $A\setminus\{z_1,\ldots,z_{k-2}\}$ where $z_1,\ldots,z_{k-2}\in A$ are distinct.  Since the diffeomorphism $f:C\to C$ preserves the ends of $C$, it follows that $f$ extends to a homeomorphism $\overline{f}:A\to A$ with $\overline{f}(z_i)=z_i$ for $i=1,\ldots,k-2$.

Let $\mu$ denote the measure on $C$ determined by $\omega$. We extend $\mu$ to a measure $\overline{\mu}$ on $A$ by setting $\overline{\mu}(U)=\mu(U\setminus\{z_1,\ldots,z_{k-2}\})$ for any Borel measurable set $U\subset A$. The homeomorphism $\overline{f}$ preserves the measure $\overline{\mu}$. The measure $\overline{\mu}$ has full support because $\mu$ does, and is non-atomic because the points $z_1,\ldots,z_{k-2}$ have measure zero. Thus by Theorem~\ref{thm:berep}, there is a homeomorphism $h:A\to A$ such that $h_*\overline{\mu}$ is a rescaling of the Lebesgue measure. In particular, the conjugate homeomorphism $h\circ \overline{f} \circ h^{-1}:A\to A$ preserves Lebesgue measure.

It now follows from Theorem~\ref{thm:franks} that the homeomorphism $\overline{f}$ has either no periodic orbits or infinitely many. Since $z_1,\ldots,z_{k-2}$ are fixed points of $\overline{f}$, Proposition~\ref{prop:usefranks} follows immediately.
\end{proof}

\subsection{Proof of the main theorem}

\begin{proof}[Proof of Theorem~\ref{thm:main}.]
Assume that $\lambda$ has only finitely many simple Reeb orbits; we need to show that $\lambda$ has exactly two simple Reeb orbits. Let $J$ be a generic $\lambda$-compatible almost complex structure on $\R\times Y$.

By Proposition~\ref{prop:special}, there exists a special $J$-holomorphic curve $C$ in $\R\times Y$. By Definition~\ref{def:special}, this implies that the following conditions hold:
\begin{description}
\item{(a)}
$C$ is irreducible and embedded and has $\op{ind}(C)=I(C)=2$.
\item{(b$'$)}
$C$ has at least two ends.
\item{(c$'$)}
$C$ has genus zero.
\item{(d)}
$C$ does not have two positive ends, or two negative ends, at covers of the same simple Reeb orbit.
\item{(e$'$)}
$C$ has no ends at positive hyperbolic orbits or at non-simple negative hyperbolic orbits.
\item{(f)}
$\M_C^J/\R$ is compact.
\end{description}
(Conditions (b$'$), (c$'$), and (e$'$) above are weaker than the corresponding conditions (b), (c), and (e) in Definition~\ref{def:special}, but are all that we need in the proof.\footnote{In fact, as explained in the discussion preceding \cite[Cor.\ 1.4]{hwz1}, the Brouwer translation theorem allows one to avoid using condition (b$'$) entirely.  However, we have kept this condition in order to streamline the exposition.})

We now check that $C$ satisfies hypotheses (i)--(v) of Proposition~\ref{prop:gss}, so that $C$ projects to a global surface of section for the Reeb flow.

(i) By (d) and (e$'$) above, $C\in\M^J(\alpha,\beta)$ where $\alpha$ and $\beta$ are admissible orbit sets.  Then by (a) and \cite[Prop.\ 3.7(2)]{bn}, every $C'\in\M^J_C$ is embedded in $\R\times Y$.

(ii) This follows from (a), (c$'$), and (e$'$) above.

(iii) This is condition (d) above.

(iv) By (a) and (d) above, and by Proposition~\ref{prop:partitions}, if $C$ has a positive end at an $m$-fold cover of a simple Reeb orbit $\gamma$, then $p_\gamma^+(m)=(m)$; and if $C$ has a negative end at an $m$-fold cover of a simple Reeb orbit $\gamma$, then $p_\gamma^-(m)=(m)$. Hypothesis (iv) now follows from Lemma~\ref{lem:relprime}.

(v) This is condition (f) above.

Thus Proposition~\ref{prop:gss} is applicable to $C$, and as in \S\ref{sec:prm} we obtain a Poincar\'{e} return map
\[
f: \pi_Y(C) \longrightarrow \pi_Y(C).
\]

By Lemma~\ref{lem:poincare}, we can apply Proposition~\ref{prop:usefranks} to the map $f$. Since we are assuming that $\lambda$ has only finitely many Reeb orbits, it follows from Proposition~\ref{prop:usefranks} and condition (b$'$) that $C$ has exactly two ends, at Reeb orbits which we denote by $\gamma_+$ and $\gamma_-$, and that $\lambda$ has no simple Reeb orbits other than the simple Reeb orbits underlying $\gamma_+$ and $\gamma_-$, which we denote by $\overline{\gamma_+}$ and $\overline{\gamma_-}$. Moreover, $\overline{\gamma_+}$ and $\overline{\gamma_-}$ are distinct by Theorem~\ref{thm:two} (one can also show this more directly using intersection theory). Thus $\lambda$ has exactly two simple Reeb orbits.
\end{proof}

\section{The non-torsion case}
\label{sec:final}

We conclude by proving Theorem~\ref{thm:nontorsion}. Below, if $\gamma_1$ and $\gamma_2$ are Reeb orbits, let $i_{\gamma_1,\gamma_2}$ denote the map
\begin{equation}
\label{eq: inclusion}
\begin{split}
i_{\gamma_1,\gamma_2}: \mathbb{Z}^2 &\longrightarrow H_1(Y),\\
(m_1,m_2) &\longmapsto m_1[\gamma_1]+m_2[\gamma_2].
\end{split}
\end{equation}

\begin{proof}[Proof of Theorem~\ref{thm:nontorsion}(b).] 
We know by Theorem~\ref{thm:wh} that there are at least three simple Reeb orbits. Suppose to get a contradiction that there are exactly three.

Choose $\Gamma \in H_1(Y)$ such that $c_1(\xi) + 2 \op{PD}(\Gamma)$ is torsion.  Since we are assuming that $c_1(\xi)$ is not torsion, it follows that $\Gamma\in H_1(Y)$ is not torsion either. 

By Proposition~\ref{prop:odd}, at least one of the simple Reeb orbits is positive hyperbolic.

We claim that the other two simple Reeb orbits are elliptic. To see this, note that if there are no elliptic orbits, then there are just three hyperbolic simple Reeb orbits, so it follows from the definition of the $ECH$ chain complex that $ECH_*(Y,\lambda,\Gamma)$ is finitely generated, contradicting Proposition~\ref{prop:Useq}(a). If there is one elliptic simple Reeb orbit $e$ and two hyperbolic simple Reeb orbits $h_1,h_2$, then let $\{\sigma_k\}_{k\ge 1}$ be a U-sequence in $ECH_*(Y,\lambda,\Gamma)$ provided by Proposition~\ref{prop:Useq}(a).
By \eqref{eqn:actionrepresentation}, the spectral invariant $c_{\sigma_k}(Y,\lambda)$ is the symplectic action of some ECH generator $\alpha_k=e^{m_k}h_1^{n_{1,k}}h_2^{n_{2,k}}$ where $m_{k}$ is a nonnegative integer and $n_{1,k},n_{2,k}\in\{0,1\}$. By \eqref{eqn:actionU}, the symplectic action of $\alpha_k$ is a strictly increasing function of $k$. It then follows that $c_{\sigma_k}(Y,\lambda)$ grows at least as $k\mc{A}(e)/4$, so that  $c_{\sigma_k}(Y,\lambda)^2/k$ grows at least linearly with $k$.  This contradicts the asymptotic formula \eqref{eqn:Useqasymptotics}.

Thus there are two elliptic simple Reeb orbits $e_1,e_2$ and one positive hyperbolic simple Reeb orbit $h$. We claim now that the kernel of $i_{e_1,e_2}$ has rank one. If the kernel of  $i_{e_{1},e_{2}}$ has rank zero, then as before $ECH(Y,\lambda,\Gamma)$ is finitely generated, a contradiction.  If the kernel has rank two, then the orbits $e_1$ and $e_2$ represent torsion classes in homology. Since by the definition of the mod 2 grading $I_2$, every generator of the chain complex $ECC_{\op{even}}(Y,\lambda)$ takes the form $e_1^{m_1}e_2^{m_2}$, we then have $ECH_{\op{even}}(Y,\lambda,\Gamma)=0$, contradicting Proposition~\ref{prop:Useq}(b).  

Thus the kernel of $i_{e_{1},e_{2}}$ has rank one, and is generated by some integer vector $(v_1,v_2)$. Without loss of generality $v_2>0$. Consider a $U$-sequence $\{\sigma_k\}_{k\ge 1}$ in $ECH_{\op{even}}(Y,\lambda,\Gamma)$. By \eqref{eqn:actionrepresentation}, for each $k$ the spectral invariant $c_{\sigma_k}(Y,\lambda)$ is the action of an orbit set $e_1^{m_{1,k}}e_2^{m_{2,k}}$ in the homology class $\Gamma$. For each $k$, we may express
\begin{align}
\label{eqn:ak}
(m_{1,k},m_{2,k})=(m_{1,1},m_{2,1})+a_k(v_1,v_2)
\end{align}
for some $a_k \in \mathbb{Z}$. By \eqref{eqn:actionU}, the $a_k$ are distinct. Since $m_{1,k}, m_{2,k} \geq 0$ by the definition of orbit set, and since $v_2>0$, it follows that $v_1 \ge 0$.  (Otherwise there could only be finitely many $k$ such that \eqref{eqn:ak} has both components nonnegative.) Since both $v_1$ and $v_2$ are nonnegative, it follows that the sequence $c_{\sigma_k}(Y,\lambda)$ grows at least linearly with $k$, since each term in this sequence exceeds the previous one by at least $\min(\mc{A}(e_1),\mc{A}(e_2))$. Once again this contradicts the asymptotic formula \eqref{eqn:Useqasymptotics}.
\end{proof}

\begin{proof}[Proof of Theorem~\ref{thm:nontorsion}(a).]

By Theorem~\ref{thm:two}, there are at least two simple Reeb orbits. Suppose to get a contradiction that there are exactly two simple Reeb orbits, and denote these by $\gamma_1$ and $\gamma_2$. Choose $\Gamma$ such that $c_1(\xi) + 2\op{PD}(\Gamma)$ is torsion. By Proposition~\ref{prop:Useq}(a), there is a $U$-sequence in the class $\Gamma$; it follows from \eqref{eqn:actionrepresentation} and \eqref{eqn:Useqasymptotics} that there is an infinite sequence $(\gamma_1^{m_{1,k}}\gamma_2^{m_{2,k}})_{k\ge 1}$ of orbit sets in the class $\Gamma$ with strictly increasing action.

Now consider the map on homology \eqref{eq: inclusion}. If the kernel of this map has rank $0$, then there is at most one orbit set in every homology class, contradicting the existence of infinitely many orbit sets in the class $\Gamma$. If the kernel has rank $2$, then there can not be any orbit sets in any non-torsion homology class, which again is a contradiction since our hypothesis that $c_1(\xi)$ is not torsion implies that $\Gamma$ is not torsion either.  If the kernel has rank $1$, then we can repeat the last paragraph of the proof of Theorem~\ref{thm:nontorsion}(b) to get a contradiction.
\end{proof}

\begin{appendix}

\section{U-sequences from Seiberg-Witten theory}
\label{app:swfacts}

We now prove Proposition~\ref{prop:Useq}. Let $\widehat{HM}^*(Y,\frak{s})$ denote Seiberg-Witten Floer cohomology with $\mathbb{Z}$ coefficients. Like $\widehat{HM}^*(Y,\frak{s}; \Z/2)$, the groups $\widehat{HM}^*(Y,\frak{s})$ have a canonical $\Z/2$ grading which refines a relative $\Z/d$ grading. This allows us to split
\[
\widehat{HM}^*(Y,\frak{s}) = \widehat{HM}^{\op{even}}(Y,\frak{s}) \oplus \widehat{HM}^{\op{odd}}(Y,\frak{s}).
\]
Since Taubes's isomorphism \eqref{eqn:taubes} preserves the relative gradings, it follows that for any given $\Gamma$, this isomorphism will either preserve or switch the decompositions of $ECH$ and $\widehat{HM}$ into even and odd parts.

The $U$ map on $\widehat{HM}^*(Y,\frak{s}; \Z/2)$ also lifts to a canonical degree $2$ map on $\widehat{HM}^*(Y,\frak{s})$. Define a ``$U$-sequence'' on $\widehat{HM}^*$ analogously to the definition in \S\ref{sec:Taubes}. By Theorem~\ref{thm:taubes} and equation \eqref{eqn:spinc}, Proposition~\ref{prop:Useq} follows from the following lemma:

\begin{lemma}
\label{lem:SW}
Let $Y$ be a closed oriented connected three-manifold, and let $\frak{s}$ be a spin-c structure on $Y$ with $c_1(\frak{s})\in H^2(Y;\Z)$ torsion. Then:
\begin{description}
\item{(a)}
There exists a $U$-sequence $\{\sigma_k\}_{k\ge 1}$ in $\widehat{HM}^{\op{even}}(Y,\frak{s})$ such that each $\sigma_k$ is non-torsion.
\item{(b)} If $b_1(Y)>0$, then there exist $U$-sequences in both $\widehat{HM}^{\op{even}}(Y,\frak{s})$ and $\widehat{HM}^{\op{odd}}(Y,\frak{s})$ such that each $\sigma_k$ is non-torsion.
\end{description}
\end{lemma}

While this lemma is well known, we present a proof for completeness.

\begin{proof}[Proof of Lemma~\ref{lem:SW}.]
As explained in \cite[\S22.3]{km}, there are companion groups $\widecheck{HM}^*(Y,\frak{s})$ and $\overline{HM}^*(Y,\frak{s})$ which fit into an exact triangle
\begin{equation}
\label{eqn:exacttriangle}
\cdots \longleftarrow \widecheck{HM}^*(Y,\frak{s})^{*} \longleftarrow \widehat{HM}^{*}(Y,\frak{s}) \longleftarrow \overline{HM}^{*-1}(Y,\frak{s}) \longleftarrow  \cdots.
\end{equation}
By construction, as in the proof of \cite[Cor.\ 35.1.4]{km}, the groups $\widecheck{HM}^*(Y,\frak{s})$ vanish in sufficiently negative degree.  Hence, by \eqref{eqn:exacttriangle}, it suffices to prove that there are such $U$-sequences in $\overline{HM}^*(Y,\frak{s})$.  By the calculations in \cite[\S 35]{km}, the latter group is a module over $\mathbb{Z}[U,U^{-1}]$.  It therefore suffices to prove that $\overline{HM}_*(Y,\frak{s})\otimes \mathbb{R}$ is non-vanishing in even degrees when $b_1(Y)=0$ and that it is non-vanishing in both even and odd degrees when $b_1(Y)>0$. 
	
By \cite[\S 35]{km}, we have \begin{align} \label{eq:twistedco} \overline{HM}_*(Y,\frak{s})\otimes \mathbb{R} \simeq H_*({\mathbb T}^{b_1(Y)},L) \otimes \mathbb{R} \end{align}
where the right hand side denotes the ``coupled Morse homology'' defined in \cite[\S33]{km} for the torus ${\mathbb T}^{b_1(Y)},$ equipped with a suitable family of self-adjoint Fredholm operators $L$. (When $b_1(Y)=0$, this torus is to be interpreted as a single point.)  By \cite[Thm.\ 34.3.1]{km}, the vector space $H_*({\mathbb T}^{b_1(Y)},L) \otimes \mathbb{R}$ is isomorphic to the homology of the twisted de Rham complex 
\[
\left(\Omega^*({\mathbb T}^{b_1(Y)})\otimes \mathbb{R}[U,U^{-1}], d + U \eta \wedge\right)
\]
where $\eta $ is a suitable closed three form.  

The rest of the proof now goes via classical topology.  By \cite[p.\ 681]{km}, the homology of the above twisted de Rham complex is computed by a spectral sequence whose $E^3$ page is
\[
H^*({\mathbb T}^{b_1(Y)}) \otimes \mathbb{R}[U,U^{-1}]
\]
with differential
\[
d_3: x \longmapsto U(\eta \wedge x) \nonumber.
\]
Furthermore, this spectral sequence degenerates after this page. Now a graded module over $\mathbb{R}[U,U^{-1}]$ may be viewed equivalently as a $\Z/2$ graded vector space over $\mathbb{R}$. Applying this to $H_*({\mathbb T}^{b_1(Y)},L) \otimes \mathbb{R}$ and taking Euler characteristics, we obtain
\[
\chi\left(H_*({\mathbb T}^{b_1(Y)},L) \otimes \mathbb{R}\right)=\chi\left(H^*({\mathbb T}^{b_1(Y)})\right).
\]

If $b_1(Y)=0$, then $\chi\left(H^*({\mathbb T}^{b_1(Y)})\right)=1$, which proves assertion (a) of the lemma in view of the isomorphism \ref{eq:twistedco}.

If $b_1(Y)>0$, then $\chi\left(H^*({\mathbb T}^{b_1(Y)})\right)=0$. Combined with \cite[Cor.\ 35.1.3]{km}, which says that $H_*({\mathbb T}^{b_1(Y)},L) \otimes \mathbb{R}$ is never vanishing, and the isomorphism  \ref{eq:twistedco}, this proves assertion (b).
\end{proof}

\end{appendix}

\end{document}